\documentclass[11pt]{amsart}
\usepackage{amssymb,amsmath,latexsym,enumerate, dsfont}
\usepackage{graphicx,verbatim,enumerate,%
ifpdf,mdwlist,xspace}
\usepackage{color}
\usepackage{mathabx}
\ifpdf
\usepackage{hyperref}
\else
\usepackage[hypertex]{hyperref}
\fi
\usepackage{graphicx}
\usepackage{cancel}
\usepackage{float}
\renewcommand{\phi}{\varphi}

\newtheorem{theorem}{Theorem}[section]
\newtheorem{thm}[theorem]{Theorem}
\newtheorem{fact}[theorem]{Fact}
\newtheorem{lemma}[theorem]{Lemma}
\newtheorem{corollary}[theorem]{Corollary}

\newtheorem{question}[theorem]{Question}
\newtheorem{observation}[theorem]{Observation}
\newtheorem{defi}[theorem]{Definition}
\newenvironment{defn}{\begin{defi} \rm}{ \end{defi}}
\newtheorem{exa}[theorem]{Example}
\newenvironment{example}{\begin{exa} \rm}{ \end{exa}}
\newenvironment{remark}{\begin{rem} \rm}{ \end{rem}}

\newtheorem{rem}[theorem]{Remark}
\newtheorem{claim}[theorem]{Claim}

\DeclareMathOperator{\range}{range}

\DeclareMathOperator{\Id}{Id}
\DeclareMathOperator{\bId}{\mathbf{Id}}

\DeclareMathOperator{\Ceers}{\mathbf{Ceers}}
\DeclareMathOperator{\Dark}{\mathbf{Dark}}
\DeclareMathOperator{\dark}{\textrm{Dark}}
\DeclareMathOperator{\Light}{\mathbf{Light}}
\DeclareMathOperator{\Code}{Code}
\DeclareMathOperator{\Edge}{Edge}

\DeclareMathOperator{\Th}{Th}

\newcommand{\ismc}{$\I$-strongly minimal cover\xspace}

\newcommand{\I}{\mathcal{I}}

\newcommand{\rel}[1]{\mathrel{#1}}

\usepackage{pifont}
\def\G1{\hbox{$\displaystyle{\mbox{\ding{172}}}$}}

\title[The Complexity of The Theory of Ceers]{The Theory of Ceers
Computes True Arithmetic}

\author[Andrews]{Uri Andrews}
\address{Department of Mathematics\\
University of Wisconsin\\
Madison, WI 53706-1388\\
USA}
\email{\href{mailto:andrews@math.wisc.edu}{andrews@math.wisc.edu}}
\urladdr{\url{http://www.math.wisc.edu/~andrews/}}

\author[Schweber]{Noah Schweber}
\address{Department of Mathematics\\
	University of Wisconsin\\
	Madison, WI 53706-1388\\
	USA}
\email{\href{mailto:ndschweber@gmail.com }{ndschweber@gmail.com}}
\urladdr{\url{http://www.math.wisc.edu/~schweber/}}

\author[Sorbi]{Andrea Sorbi}
\address{Dipartimento di Ingegneria Informatica e Scienze Matematiche\\
Universit\`a Degli Studi di Siena\\
I-53100 Siena, Italy}
\email{\href{mailto:andrea.sorbi@unisi.it}{andrea.sorbi@unisi.it}}
\urladdr{\url{http://www3.diism.unisi.it/~sorbi/}}

\thanks{
Andrews was partially supported by NSF grant DMS1600228. Andrews and Sorbi were partially supported by a grant of the Science
Committee of the Republic of Kazakhstan, grant number AP05131579.
Sorbi is a member of INDAM-GNSAGA. Sorbi's research was partially supported by PRIN 2017 Grant
``Mathematical Logic: models, sets, computability''
}

\keywords{Computably enumerable equivalence relation; computable
reducibility on equivalence relations.}

\subjclass[2010]{03D25}

\begin{document}

\begin{abstract}
We show that the theory of the partial order of computably enumerable
equivalence relations (ceers) under computable reduction is 1-equivalent to
true arithmetic. We show the same result for the structure comprised of the
dark ceers and the structure comprised of the light ceers. We also show the
same for the structure of $\I$-degrees in the dark, light, or complete
structure. In each case, we show that there is an interpretable
copy of $(\mathbb{N},+,\cdot)$.
\end{abstract}

\maketitle

\section{Introduction}\label{sct:introduction}
A major theme in investigating computability theoretic reducibilities has
been to measure, and when possible to characterize, the complexity of the
first order theory of their degree structures. Throughout the paper we regard
a degree structure as a poset, and if $\mathcal{P}$ is a poset then the
\emph{theory of $\mathcal{P}$}, denoted by $\Th(\mathcal{P})$, is the set of
sentences, in the first order language of posets, that are true in
$\mathcal{P}$.
Investigating the theory of a degree structure is an important task, not only because it sheds relevant
information about the corresponding reducibility, but also because it generally stimulates
useful techniques and constructions which are developed for this purpose.

Typically, a reducibility is a binary relation $\leq$ on subsets of the set
$\omega$ of natural numbers, which gives rise to a degree structure
$\mathcal{D}$ which partitions the power set $P(\omega)$ into equivalence
classes called degrees. Let $\nu: P(\omega) \rightarrow
\mathcal{D}$ be the surjection so that $\nu(X)$ is the degree of $X$, and
$\leq$ denote the partial ordering relation on $\mathcal{D}$. The relations
(in $X,Y$) $\nu(X)=\nu(Y)$ and $\nu(X)\leq \nu(Y)$ are typically arithmetical, so that
one can effectively translate first order sentences regarding degrees to
second order sentences of arithmetic. This yields a reduction
$\Th(\mathcal{D})\leq_1 \Th^2(\mathbb{N})$, where the latter symbol denotes
true second order arithmetic. In many cases, it is true that these two
theories are 1-equivalent, and the challenge is to show that the reverse
reduction, i.e. $\Th^2(\mathbb{N}) \leq_1 \Th(\mathcal{D})$, holds as well. By the Myhill Isomorphism Theorem, the two theories are
then computably isomorphic, and thus
$\Th(\mathcal{D})$ is as complicated as it
can be.  The literature here is indeed rich of classical and celebrated
results, starting from Simpson~\cite{Simpson} who showed that the theory of
the Turing degrees is computably isomorphic to $\textrm{Th}^2(\mathbb{N})$.
See also \cite{Slaman-Woodin1}. To mention two other major reducibilities,
the theory of the $m$-degrees (\cite{Nerode-Shore}) and the theory of the
enumeration degrees (\cite{Slaman-Woodin2}, see also \cite{CAi-totality}) are
also computably isomorphic to $\textrm{Th}^2(\mathbb{N})$.

When no restriction is taken on the universe of the reducibility, then one
talks about the \emph{global} degree structure of that reducibility. It is
common however to consider \emph{local} degree structures as well, by
restricting attention to special countable families of degrees. This is the
case for instance (just to consider some local structures of the
aforementioned global structures) of the Turing degrees below the first jump,
or the Turing degrees of the computably enumerable (c.e.) sets, or the
$m$-degrees of the c.e. sets, or the enumeration degrees below the first
enumeration jump. If $\mathcal{D}$ is a local structure then typically one
finds a surjection $\nu:\omega \longrightarrow \mathcal{D}$ such that the
relations (in $x,y$) $\nu(x)=\nu(y)$ and $\nu(x)\leq \nu(y)$ are
arithmetical, so that one can effectively translate sentences on $\mathcal{D}$
into first order arithmetical sentences, establishing a reduction
$\Th(\mathcal{D})\leq_1 \Th^1(\mathbb{N})$, where the latter set denotes true
first order arithmetic $\Th(\mathbb{N}, +, \times)$. To show that
$\Th(\mathcal{D})$ is as complicated as possible (i.e. computably isomorphic
to $\Th^1(\mathbb{N})$) it is then enough to show $\Th^1(\mathbb{N}) \leq_1 \Th(\mathcal{D})$. For instance,
this has been done for the aforementioned local structures: for the Turing
degrees below the first jump, see Shore~\cite{Shore}, for the c.e. Turing
degrees see Nies, Shore and Slaman~\cite{Nies-Shore-Slaman}, for the c.e.
$m$-degrees, see Nies~\cite{Nies-last-question}, for the enumeration degrees
below the first enumeration jump, see Ganchev and
Soskova~\cite{Ganchev-Soskova1, Ganchev-Soskova2}.

The above examples are about reducibilities on sets of natural numbers. We
consider in this paper a reducibility on equivalence relations on $\omega$,
instead of subsets of $\omega$. The reduction is defined as follows: if $R,S$
are equivalence relations on $\omega$, we say that $R$ is \emph{computably
reducible} (or, simply, \emph{reducible}) to $S$ (notation: $R \leq S$) if
there is a computable total function $f$ such that
\[
(\forall x,y)[x  \rel{R} y \Leftrightarrow f(x) \rel{S} f(y)].
\]
As with other reducibilities, $\leq$ gives rise to an equivalence relation
$\equiv$, where $R \equiv S$ if $R \leq S$ and $S \leq R$. The equivalence
class of an equivalence relation $R$ under $\equiv$ will be called the
\emph{degree of $R$}. The first study of computable reducibility on
equivalence relations on natural numbers was initiated by Ershov \cite{Ershov:positive, Ershov:NumberingsI} in the 1970s. Recently, there has
been a revived interest in this reducibility, motivated by considering computable reducibility as an effective version of Borel reducibility on equivalence relations \cite{Fokina2010effective}. Borel reducibility on equivalence relations is a primary target of interest in descriptive set theory, see for
instance~\cite{Becker-Kechris:descriptive}. The revived interest in computable reducibility has also been caused by its high potential as a
tool for measuring the computational complexity of classification problems in computable mathematics. For
instance it is shown in \cite{Fokina-et-al-several} that the isomorphism
relation for various familiar classes of computable groups is
$\Sigma^1_1$-complete under this reducibility.

The global structure of the degrees of equivalence relations, however, has
not been extensively studied. Much more attention has been given to its local
structure consisting of the degrees of the c.e. equivalence
relations (commonly called \emph{ceers} after \cite{Gao-Gerdes}). Indeed,
ceers have played a leading role in the tale of computable reducibility as they
appeared as the main characters of what are perhaps the first results about
$\leq$ (although before the notion appeared in the literature). Namely,
Miller~III \cite{Miller} constructed a finitely presented group whose
word problem is $\Sigma^0_1$-complete with respect to $\leq$, and
Miller~III \cite{Miller} proved that the isomorphism problem for finitely
presented groups is $\Sigma^0_1$-complete with respect to $\leq$. For other
applications of $\leq$ to word problems of finitely presented groups see
\cite{Nies-Sorbi}. The first work explicitly tackling $\leq$ on ceers was
done by Ershov~\cite{Ershov:positive}, who pointed out important examples of
$\Sigma^0_1$-complete ceers, showing also that their degree is
join-irreducible. In the 1980s, the reducibility $\leq$ on ceers was applied to
study computability theoretic properties of the relation of provable
equivalence of sufficiently expressive formal systems,
see for instance
\cite{Bernardi-Sorbi, Montagna-ufp}. Additional interest for
computable reducibility on ceers comes from the study of c.e.\ presentations
of structures, as is examined for instance in
\cite{fokina2016linear,gavruskin2014graphs}. It is also worth noticing that
ceers have been investigated in computability theory also not in connection
with computable reducibility. For instance, Carroll~\cite{Carroll} studied
the lattice of ceers under inclusion, and  Nies~\cite{Nies-eq-rel-mod-finite}
studied ceers modulo finite differences. They both showed that the first
order theory of the resulting structures is computably isomorphic to
$\textrm{Th}^1(\mathbb{N})$.

More explicitly oriented toward a degree-theoretic approach are the papers on
ceers by Gao and Gerdes~\cite{Gao-Gerdes},
Andrews~et~al.~\cite{Andrews-et-al}, \cite{jumps}, and finally Andrews and
Sorbi~\cite{Andrews-Sorbi}. The last paper provides a thorough investigation
of the structure $\Ceers$ comprised of $\equiv$-degrees of ceers under the reducibility $\leq$, with emphasis on existence and non-existence of
meets and joins, minimal covers, definable classes of degrees, and
automorphisms.

It was shown in
\cite{Andrews-et-al} that $\Th(\Ceers)$ is undecidable, and even the
$\forall\exists\forall$-fragment (again, in the language of posets) is undecidable.
In this paper we completely characterize the
complexity of $\Th(\Ceers)$ by showing that it is in fact as complicated as
it can be, namely computably isomorphic to $\textrm{Th}^1(\mathbb{N})$. We do
this also for two suborders of $\Ceers$ introduced in \cite{Andrews-Sorbi}, called $\Light$ and $\Dark$  which are defined in the next section,
and also for the quotient structures obtained by quotienting the three structures
$\Ceers$, $\Light$, and $\Dark$ modulo uniform joins with finite ceers. In
each case, we show that there is an interpretable copy of
$(\mathbb{N},+,\cdot)$ in the degree structure.

\subsection{Outline of proof}\label{outlineofproof}

In section \ref{Coding Using Parameters}, we give a definable copy of $(\mathbb{N},+,\cdot)$ using a parameter in $\Dark/\I$. To do this, we show
that we can encode any graph (modulo some isolated vertices) in a particular way. In particular, for any $\Dark/I$-degree $\boldsymbol{c}$, we let the universe of $G_{\boldsymbol{c}}$ be the collection of minimal $\I$-degrees $\leq \boldsymbol{c}$. We say that $\boldsymbol{a}$ is an $\I$-strongly minimal cover of a pair of incomparable $\I$-degrees $\boldsymbol{d}$ and $\boldsymbol{e}$ if the degrees $<_\I \boldsymbol{a}$ are precisely the degrees $\leq_{\I}\boldsymbol{d}$ and the degrees $\leq_{\I}\boldsymbol{e}$. We put an edge in $G_{\boldsymbol{c}}$ between degrees $\boldsymbol{d}$ and $\boldsymbol{e}$ if there are two incomparable degrees $\leq \boldsymbol{c}$ which are $\I$-strongly minimal covers of $\boldsymbol{d}$ and $\boldsymbol{e}$. Section \ref{sct:minimalIdark} is devoted to proving technical details about minimal $\I$-degrees and their $\I$-strongly minimal covers. The main purpose of these details are to allow us in section \ref{Coding Using Parameters} to construct, for any given computable graph $G$, a dark $\I$-degree $\boldsymbol{c}$ so that $G_{\boldsymbol{c}}$ is isomorphic to a disjoint of $G$ with a graph that has no edges. In particular, since countable graphs with no isolated vertices are universal for all countable structures under definability, this gives us a coded copy of $(\mathbb{N},+,\cdot)$ using a single parameter $\boldsymbol{c}$.

In section \ref{sct:dark-arithmetic}, we remove the dependence on the parameter to define a copy of $(\mathbb{N},+,\cdot)$. We do this as follows: We fix an interpretation of $(\mathbb{N},+,\cdot)$ in a graph. We first isolate the collection of degrees $\boldsymbol{c}$ for which $G_{\boldsymbol{c}}$ interprets (via this fixed interpretation) a model of Robinson's system $Q$
of arithmetic. We call these good codes. Next, we consider the collection of pairs $(\boldsymbol{c},\boldsymbol{d})$ where $\boldsymbol{c}$ is a good code and $\boldsymbol{d}$ is in $\boldsymbol{c}$'s encoded model of $Q$. Our goal is to isolate the collection of such pairs where $\boldsymbol{d}$ is in the standard part of the model. To do this, we define an encoding of functions inside $\Dark_{/\I}$ (via our encoding of graphs) and define an equivalence relation on such pairs $(\boldsymbol{c},\boldsymbol{d})\sim (\boldsymbol{c}',\boldsymbol{d}')$ if and only if there is an encoded order-preserving bijection between the interval
$[\boldsymbol{0},\boldsymbol{d}]$ in $\boldsymbol{c}$'s model of $Q$ and the interval $[\boldsymbol{0},\boldsymbol{d}']$ in $\boldsymbol{c'}$'s model of $Q$. In other words, $\boldsymbol{d}$ and $\boldsymbol{d}'$ code the same number in their respective models. Not every function can be encoded in the structure, since it is countable, but our encoding is sufficient to capture every finite function on our models of $Q$. Much of section \ref{sct:dark-arithmetic} is devoted to this encoding. Finally, we define $\mathcal{N}$ to be the collection of pairs $(\boldsymbol{c},\boldsymbol{d})$ where $\boldsymbol{c}$ is a good code and for every good code $\boldsymbol{c}'$, there is some $\boldsymbol{d}'$ so that $(\boldsymbol{c},\boldsymbol{d})\sim (\boldsymbol{c}',\boldsymbol{d}')$. Since there is some good code which encodes a copy of $(\mathbb{N},+,\cdot)$, and our collection of encoded functions includes all the finite functions on our models of $Q$, this isolates the $\sim$-classes of the standard naturals within the various models of $Q$. Finally, we can define the operations $+$ and $\cdot$ on $\mathcal{N}$ which makes it isomorphic to $(\mathbb{N},+,\cdot)$.

The above gives an interpreted copy of $(\mathbb{N},+,\cdot)$ in $\Dark/\I$. Since $\I$-equivalence is $\emptyset$-definable
(i.e., definable without parameters)
in $\Dark$, and $\Dark$ is $\emptyset$-definable in $\Ceers$, this shows the result for $\Dark/\I$, $\Dark$, and $\Ceers$. We also here observe that the same encoding works in $\Ceers/\I$.

Note that in both the $\Light_{/\I}$ degrees and the $\Light$ degrees, there is a single minimal degree. Thus our encoding needs to be changed to consider $\Light_{/\I}$-degrees which are minimal over the degree of $\Id$. In section \ref{LightGraphs} we show that with this slightly altered encoding, we can again encode every computable graph modulo some isolated vertices in the $\Light_{/\I}$-degrees using a parameter. In section \ref{LightWithoutParameters}, we repeat
the strategy used in section~5
to remove the dependence on the parameter in the $\Light_{/\I}$-degrees. In this setting, the simple encoding of functions used in section \ref{sct:dark-arithmetic} does not suffice, so we introduce a more complicated encoding of functions. Aside from the different encoding of functions, the strategy is the same and again we encode our copy of $(\mathbb{N},+,\cdot)$ in $\Light_{/\I}$. Since $\I$-equivalence is $\emptyset$-definable in $\Light$, this also yields the result for $\Light$.

\section{Background material}
For more information and details about unexplained computability theoretic
terminology or results exploited in the paper without any reference, the
reader may consult any standard textbook, see e.g.
\cite{Rogers:Book,Soare:Book}. In this section we review some background
material concerning ceers and computable reducibility. The $\equiv$-degree of
an equivalence relation $R$ will be denoted by $\deg(R)$.

\subsection{The classes $\I, \Light$, and $\Dark$}
We recall the following partition of ceers, introduced and studied in
\cite{Andrews-Sorbi}. Let $R$ be a ceer:
\begin{itemize}
  \item Let $\Id_n$ be the equivalence relation given by $x\Id_n y$ if and only if $x\equiv y\, (\text{mod}\, n)$.
  \item Let $\Id$ be the equivalence relation given by $x\Id y$ if and only if $x=y$.
  \item $R$ is \emph{finite} if it has only finitely many equivalence
      classes. Note that for each $n\geq 1$, $R$ has exactly $n$ classes if and only if $R \equiv \Id_n$.
  \item $R$ is \emph{light} if  there is an infinite c.e. set $W$
       such that $x \cancel{\rel{R}}
      y$ for each pair of distinct $x,y \in W$
      (following \cite{Gao-Gerdes}, a set with this property will be called a \emph{transversal for $R$}). Equivalently, a
      ceer $R$ is light if and only if $\Id \leq R$.
     \item $R$ is \emph{dark} if it is neither finite nor light.
\end{itemize}
The symbols $\I$, $\Light$, $\Dark$ denote the classes of finite ceers, light
ceers, and dark ceers, respectively. These classes partition the ceers, and
give rise to a corresponding partition of the degrees of ceers into three
classes of degrees (still denoted by $\I, \Light, \Dark$). $\I$ is an initial
segment of $\Ceers$ having order type $\omega$. In $(\Ceers,\leq)$, the degree of $\Id$, and thus each of these three classes are
first order $\emptyset$-definable \cite[Corollary 8.1]{Andrews-Sorbi}. $\Ceers, \Light$, and $\Dark$ are
neither upper nor lower semilattices. In this regard, the most spectacular
case is provided by dark degrees, as no pair of incomparable dark degrees has
either meet or join in $\Ceers$ or in $\Dark$ \cite[Theorems 5.8 and 7.18]{Andrews-Sorbi}.

\subsection{Some general facts about ceers}
We describe three constructions of new ceers starting from given ceers and/or
c.e. sets.

The first construction is the \emph{uniform join} $R \oplus S$ which is the
equivalence relation which copies $R$ on the even numbers and $S$ on the odd
numbers: $x \rel{R \oplus S} y$ if there exist $u,v$ such that $x=2u, y=2v$
and $u\rel{R} v$, or $x=2u+1, y=2v+1$ and $u\rel{S} v$. This operation
extends in the obvious way to any countable number of equivalence relations,
see Section~2.1 of \cite{Andrews-Sorbi}.

The second construction is described in detail in Section~2.3 of
\cite{Andrews-Sorbi}:

\begin{defn}\label{def:restrictions}
If $E$ is a ceer and $W$ is a non-empty c.e. set then $E\restriction{W}$ (called the
\emph{restriction of $E$ to $W$}) is the ceer $x \rel{E\restriction{W}} y$ if
and only if $h(x) \rel{E} h(y)$, where $h: \omega \rightarrow  W$ is any
computable surjection (up to $\equiv$ the definition does not depend on the
chosen $h$).
\end{defn}

\begin{remark}\label{rem:restrictions}
It is clear that $h$ provides a reduction $E\restriction W \leq E$, which we
call the \emph{inclusion} of $E\restriction W$ into $E$. If $X \leq R$ via a
reduction $f$ then $X\equiv R\restriction W$, where $W=\range(f)$.
\end{remark}

\begin{fact}\label{fct:restriction}
If $X\leq R_1\oplus R_2$ then there are ceers
$X_1\leq R_1$ and $X_2\leq R_2$ such that $X \equiv X_{1}\oplus X_{2}$.
\end{fact}

\begin{proof}
If $f$ is a reduction for $X\leq R_1\oplus R_2$ and $W=\range(f)$ then
$X\equiv (R_1\oplus R_2)\restriction W\equiv R_{1}\restriction V_{1} \oplus
R_{2}\restriction V_{2}$, where $V_{1}=\{x\mid  2x \in W\}$ and
$V_{2}=\{x\mid 2x+1 \in W\}$.
\end{proof}

The third construction is described in the following definition.

\begin{defn}\label{fct:quotients}
Let $W \subseteq \omega^2$. We denote by $E_{/W}$ the equivalence relation
generated by the set of pairs $E\cup W$. If $W$ is a singleton, say
$W=\{(x,y)\}$, then we simply write $E_{/(x,y)}$ instead of $E_{/\{(x,y)\}}$.
\end{defn}

Notice that if $W$ is c.e. and $E$ is a ceer then $E_{/W}$ is a ceer as well.

We will also make use of the following easy facts about ceers:

\begin{fact}\label{fact:trivial}
The following hold:
\begin{enumerate}
  \item For every  pair $S,T$ of ceers and $k\geq 0$, $S\equiv T$ if and
      only if $S\oplus \Id_k \equiv T\oplus \Id_k$
      (where for any equivalence relation $R$, we let $R\oplus \Id_0=R$).
  \item If $f$ is a reduction from $R$ to $S$ which omits exactly $k$
      $S$-classes  then $R\oplus \Id_k \equiv
      S$.
   \item If $W$ is a non-empty
   c.e. set missing exactly $k$ equivalence classes of a
       ceer $R$ then $R\restriction W \oplus \Id_{k} \equiv R$.
      \end{enumerate}
\end{fact}

\begin{proof}
The first item is essentially \cite[Lemma~2.1]{Andrews-Sorbi}. The second
item comes from Lemma~2.8 in \cite{Andrews-Sorbi}; the third item follows
from the previous one and the fact that under the assumptions the inclusion
reduction $R\restriction W  \leq R$ misses exactly $k$ equivalence classes.
\end{proof}

\subsection{Reducibility modulo $\I$}
We recall the following reducibility from \cite{Andrews-Sorbi}.

\begin{defn}
We say $R\leq_\I S$ if there is some $k$ so that $R \leq
S\oplus \Id_k$.

Define $R \equiv_\I S$ if $R \leq_\I S$ and $S
\leq_\I R$. The equivalence class of $R$ under $\equiv_\I$ will be denoted by
$\deg_\I(X)$ and called the \emph{$\I$-degree of $R$}. $\Ceers_{/\I}$ is the
collection of all $\deg_\I(R)$, where $R$ is a ceer.	
\end{defn}

\begin{fact}\label{fct:restrictionI}
If $X\leq_{\I}R_1\oplus R_2$ then there are ceers $X_1\leq R_1$ and $X_2\leq
R_2$ so that $X\equiv_{\I} X_1\oplus X_2$.
\end{fact}

\begin{proof}
There is an $n$ so that $X\leq R_1\oplus R_2\oplus \Id_n$. By Fact
\ref{fct:restriction}, this shows that there are $X_1\leq R_1 $ and $X_2\leq
R_2$ and $F\leq \Id_n$ so that $X\equiv X_1\oplus X_2\oplus F$. Then
$X\equiv_{\I} X_1\oplus X_2$.
\end{proof}

We consider the six structures: $\Ceers$, $\Dark$, $\Light$, $\Ceers_{/\I}$,
$\Dark_{/\I}$, $\Light_{/\I}$. For elementary differences between these
classes, and for more on their structural properties, see
\cite{Andrews-Sorbi}. We note:

\begin{lemma}[\cite{Andrews-Sorbi}, Obs. 9.7] \label{lem:I-definability}
In each of $\Ceers$, $\Dark$, and $\Light$, the equivalence relation
$\equiv_\I$, and thus also the partial order $\leq_\I$, is $\emptyset$-definable.
\end{lemma}

Notice also:

\begin{lemma}\label{lem:one-way}
If $\mathcal{D}$ is any of the six structures $\Ceers$, $\Dark$, $\Light$,
$\Dark_{/\I}$, $\Light_{/\I}$, and $\Ceers_{/\I}$ then
$\Th(\mathcal{D})\leq_1 \Th^1(\mathbb{N})$.
\end{lemma}

\begin{proof}
Let $\{R_z \mid z \in \omega\}$ be the indexing of ceers defined in
\cite{Andrews-et-al}, and let $\nu: \omega \longrightarrow \Ceers$ be given
by $\nu(x)=\deg(R_x)$. Then it is easy to see that the relations (in $x,y$)
$\nu(x)=\nu(y)$ and $\nu(x)\leq \nu(y)$ are arithmetical, so that one can
effectively translate sentences on degrees into first order arithmetical
sentences, showing $\Th(\Ceers)\leq_1 \Th^1(\mathbb{N})$. Since each of the
other five structures are interpretable without parameters in $\Ceers$ by \cite[Corollary 8.1]{Andrews-Sorbi} and Lemma \ref{lem:I-definability}, this shows that their
theories also are $\leq_1 \Th^1(\mathbb{N})$.
\end{proof}

For each of the structures of degrees of ceers mentioned in
Lemma~\ref{lem:one-way}, we show that there is a copy of
$(\mathbb{N},+,\cdot)$ which is interpreted without parameters in the structure, and thus by
Lemma~\ref{lem:one-way}, the theory is $1$-equivalent to true arithmetic.  In
view of Lemma~\ref{lem:I-definability}, to yield the result for all six structures we need only find
an interpreted copy of $(\mathbb{N},+,\cdot)$ in the three structures
$\Dark_{/\I}$, $\Light_{/\I}$, and $\Ceers_{/\I}$.

\subsection{Self-full ceers}
The following obvious facts about dark ceers hold:

\begin{fact}\label{fct:dark-closure}
The following hold:
\begin{enumerate}
\item if $S$ is not finite, $S\leq_\I T$ and $T$ is dark then so is $S$;
\item if $R$ is dark and $S$ is either finite or dark then $R\oplus S$ is dark;
\item if $R$ is dark then so is any
$R_{/ W}$. In particular, if $R,S$ are dark then so is any $(R\oplus
S)_{/(2x,2y+1)}$.
\end{enumerate}
\end{fact}

\begin{proof}
Item (1)  is \cite[Observation~6.3]{Andrews-Sorbi}; item (2) is
\cite[Observation~3.2]{Andrews-Sorbi}; item (3) follows straight from the
definitions.
\end{proof}

\begin{defn}
A ceer $R$ is self-full if $R\oplus \Id_1\not\leq R$. Equivalently
(\cite[Observation~4.2]{Andrews-Sorbi}), $R$ is self-full if whenever
$\phi:R\rightarrow R$ is a reduction, $\phi$ is onto the classes of $R$, i.e., the range of $\phi$ intersects every $R$-class.
\end{defn}

The following fact collects useful properties of self-full ceers and the
$\oplus \Id_1$ operation:

\begin{fact}\label{fact:sf-useful}
The following hold:
\begin{enumerate}
  \item  For any ceers $R$ and $S$, if $S< R \oplus \Id_1$ then $S\leq R$. For any ceers $R$ and $S$, if $R< S$ then $R \oplus \Id_1 \leq S$.
  \item Every dark ceer is self-full.
\end{enumerate}

\end{fact}

\begin{proof}
Item (1) is \cite[Lemma~4.5]{Andrews-Sorbi}, and item (2) is
\cite[Lemma~4.6]{Andrews-Sorbi}.
\end{proof}

It is also useful to note:

\begin{observation}\label{obs:not-omitting}
If $R$ is self-full and $R\equiv E$, then any reduction $\phi:R\rightarrow E$
must be onto the classes of $E$.
\end{observation}

\begin{proof}
Consider the pair of reductions $R\leq E \leq R$. If we were to have a
reduction of $R$ into $E$ which is not onto the classes of $E$, then the
composition would be a reduction of $R$ to itself which is not onto the
classes of $R$, contradicting $R$ being self-full.
\end{proof}

\section{Minimal classes in $\Dark_{/\I}$}\label{sct:minimalIdark}
We now proceed with examining some facts about dark minimal $\I$-degrees, on
which we will code models of arithmetic. There are two types of minimal
degrees in $\Dark_{/\I}$. The first is the $\I$-degree of a ceer $D$ which is
a dark minimal ceer. The second is an $\I$-degree which, as a class of
$\equiv$-degrees, has the order type of $\mathbb{Z}$ and bounds no other
non-zero $\I$-degree. Indeed, suppose that $R$ is a dark ceer, and
$\deg_{\mathcal{I}}(R)$ is minimal. First of all notice that the ceers
$\{R_k\mid  k \in \omega\}$, where $R_k=R\oplus \Id_k$, give rise, as $k$
strictly increases, to a strictly increasing sequence of $\equiv$-degrees
within $\deg_{\mathcal{I}}(R)$ (where we let $R_0=R$), which comprises
all $\equiv$-degrees within $\deg_{\mathcal{I}}(R)$ that are greater than
$\deg(R)$: this easily follows from
Facts~\ref{fct:dark-closure}~and~\ref{fact:sf-useful}. On the other hand, if
$S< R$ in $\deg_{\mathcal{I}}(R)$ then there is $k\ge 1$ such that $R\equiv S
\oplus \Id_k$, and all $T$ such that $R\equiv T\oplus \Id_k$ are $\equiv S$
(using Fact~\ref{fact:trivial})(1), so for every $k\ge 1$ for which there is
an $S$ as above we can choose such a (unique up to $\equiv$) $S$ and define
$R_{-k}=S$. If there is no $\leq$-minimal element in $\deg_{\mathcal{I}}(R)$
then all ceers in $\deg_{\mathcal{I}}(R)$ which are smaller than $R$ are
$\equiv$ to some $R_{-k}$, which form a chain $\cdots < R_{-k} < \cdots
<R_{-1}$; thus the $\mathcal{I}$-degree $\deg_{\mathcal{I}}(R)$ consists of a
$\mathbb{Z}$-chain of $\equiv$-degrees, namely the degrees of
\[
\cdots < R_{-2} < R_{-1} < R < R_{1} <R_{2} < \cdots.
\]

\begin{defn}
A ceer in this second type of minimal $\I$-degree we call a
\emph{$\mathbb{Z}$-dark minimal ceer}.
\end{defn}

\begin{example}
Examples of dark minimal degrees are from Theorem~4.10 and~Corollary~4.15 of
\cite{Andrews-Sorbi}. As to examples of $\mathbb{Z}$-dark minimal ceers, they
come from Theorem~4.11 and~Corollary~4.15 of \cite{Andrews-Sorbi}: indeed,
one can use the proof of Theorem~4.11, but without coding any ceer $A$, to
build $\leq_\I$-minimal dark ceers with finite classes. That the construction
in Theorem 4.11 suffices is justified by Lemma
\ref{lem:characterization-I-minimality} below.
\end{example}

The following fact was shown in \cite{KazakhPaper}, the right-to-left
implication being already in \cite[Lemmas~3.4 and~3.5]{Andrews-Sorbi}:

\begin{fact}\label{darkMinimalCombinatorially}
$R$ is a dark minimal ceer if and only if $R$ has infinitely many classes and
every c.e. set $W$ which intersects infinitely many $R$-classes intersects
every $R$-class.
\end{fact}

An easy consequence of this fact is the following:

\begin{observation}\label{inseparabilityDarkMinimal}
If $R$ is a dark minimal ceer, then every pair of classes $[a]_R \neq [b]_R$
are computably inseparable. As a consequence, if $R$ is a dark minimal ceer
and $R\leq_\I S$ then $R\leq S$.
\end{observation}

\begin{proof}	
Suppose that $X$ is a computable set which separates $[a]_R$ and $[b]_R$.
Either $X$ or the complement of $X$ must intersect infinitely many classes in
$R$. But neither can intersect every class, since $[a]_R \subseteq X$ and
$[b]_R \cap X=\emptyset$. This contradicts
Fact~\ref{darkMinimalCombinatorially}. The latter claim follows from the fact
that if $R$ is a dark minimal ceer then no reduction
$R \leq S \oplus \Id_k$, for any $k\ge 1$,
can hit $\Id_k$ since no $R$-equivalence class is computable.
\end{proof}

In the next lemma we show the corresponding fact for the $\mathbb{Z}$-dark
minimal ceers.

\begin{lemma}\label{lem:characterization-I-minimality}
A ceer $R$ is a $\mathbb{Z}$-dark minimal ceer if and only if it has infinitely
many computable classes and every c.e. set $W$ which intersects infinitely
many $R$-classes intersects co-finitely many $R$-classes.
\end{lemma}

\begin{proof}
Suppose that $R$ is a $\mathbb{Z}$-dark minimal ceer. For every $k$, $R
\equiv S\oplus \Id_k$ for some $S$. By self-fullness of $R$ and
Observation~\ref{obs:not-omitting}, the reduction of $R$ to $S\oplus \Id_k$
is onto the classes of $S\oplus \Id_k$, so $R$ has at least $k$ computable
classes. This is true for every $k$, so $R$ has infinitely many computable
classes. Suppose $W$ is a c.e. set which intersects infinitely many
$R$-classes. Since $R\restriction W\leq R$, we see by $\I$-minimality of $R$
that $R\restriction W \oplus \Id_k\equiv R$ for some $k\ge 0$, where
again we let $R\restriction W \oplus \Id_0=R\restriction W$. Suppose now that $W$ omits
$\ge k+1$ classes. Then from the inclusion $R\restriction W \leq R$ and
sending the $k$-classes of $\Id_k$ to $k$ of the classes omitted by $W$ it is
possible to build a reduction $R\restriction W \oplus \Id_k \leq R$ which
omits at least one class in its range. But this, together with $R \leq
R\restriction W \oplus \Id_k$, contradicts Observation~\ref{obs:not-omitting}.
It follows that $W$ hits co-finitely many classes.

For the converse: First we check that, under the assumptions, $R$ is dark.
Suppose $W$ enumerated an infinite transversal. Consider $V$ any c.e.
co-infinite subset of $W$. Since $W$ is a transversal, $V$ cannot contain
co-finitely many classes in $R$. Yet $V$ hits infinitely many classes in $R$,
contradicting the hypothesis on $R$. Thus $R$ is dark.

Since $R$ has infinitely many computable classes, then for every $k$ we can
collapse any $k+1$ of these together to find an $E$ so that $E\oplus
\Id_k\equiv R$. Thus $R$ is not in the $\I$-class of a dark minimal ceer, for
choosing such an $E$ and taking $R_{-k}=E$,
we must have an infinite strictly descending chain provided by the
$\equiv$-degrees of the various $R_{-k}$. Towards showing $\I$-minimality of $R$, suppose $X\leq_\I R$, where $X$ is not finite. By Fact \ref{fct:restrictionI},
$X\equiv_{\I} Y$ for some  $Y\leq R$. By Remark \ref{rem:restrictions}, there
is some $W$ so that $Y\equiv R\restriction W$. Since $X$ is not finite, $Y$
must not be finite, and thus $W$ intersects infinitely many $R$-classes. Thus
$W$ intersects co-finitely many $R$-classes. Thus, by Fact
\ref{fact:trivial}(3), $Y\equiv R_{-m}$ for some $m$. Thus, $R\leq_{\I}
R_{-m}\leq_{\I} X$. Thus $R$ is $\I$-minimal.
\end{proof}

We isolate the following fact which is immediate from
Lemma~\ref{lem:characterization-I-minimality}:

\begin{corollary}\label{cor:Z-dark-infy-computable}
If $R$ is a $\mathbb{Z}$-dark minimal ceer, $W$ is c.e., and $W$ intersects infinitely many
classes, then the $R$-closure of $W$ is computable, and every class omitted from
$W$ is computable.
\end{corollary}

\begin{proof}
The $R$-closure of $W$ along with each of the finitely many omitted classes forms
a partition of $\omega$ into finitely many c.e. sets. Thus, each of these
sets is computable.
\end{proof}

\begin{remark}\label{rmk:co-finitelyMany}
Notice that although in Fact~\ref{darkMinimalCombinatorially} and
Lemma~\ref{lem:characterization-I-minimality} we distinguish between the
cases of being dark minimal and being $\mathbb{Z}$-dark minimal for a ceer
lying in a dark minimal $\I$-degree, it is correct to say, unifying the two
cases, that if $R$ lies in a dark minimal $\I$-degree then every c.e. set
intersecting infinitely many $R$-equivalence classes intersects co-finitely
many classes. If $R$ is in the $I$-degree of a dark minimal ceer, say $R\equiv R_0\oplus \Id_n$ where $R_0$ is dark minimal, then ``co-finitely many classes'' means
in fact in this case ``all classes except for the $n$ computable classes''.
\end{remark}

\begin{defn}
We say that $\boldsymbol{a}$ is an \emph{\ismc} of the pair of
$\mathcal{I}$-degrees $\boldsymbol{d},\boldsymbol{e}$ if $\boldsymbol{d}$ and
$\boldsymbol{e}$ are $\I$-incomparable, $\boldsymbol{d},\boldsymbol{e}<_\I
\boldsymbol{a}$, and the only degrees $<_\I \boldsymbol{a}$ are $\leq_\I
\boldsymbol{d}$ or $\leq_\I \boldsymbol{e}$.
\end{defn}

\begin{lemma}\label{oplusISMC}
If $\boldsymbol{r}_1$ and $\boldsymbol{r}_2$ are distinct dark minimal
$\I$-degrees, then $\boldsymbol{r}_1\oplus \boldsymbol{r}_2$ is an \ismc of
the pair $\boldsymbol{r}_1,\boldsymbol{r}_2$.
\end{lemma}

\begin{proof}
Let $R_1$ and $R_2$ be in the $\I$-degrees $\boldsymbol{r}_1,
\boldsymbol{r}_2$ respectively, and suppose $X<_{\I} R_{1} \oplus R_{2}$. We
want to show that either $X\leq _{\I} R_{1}$ or $X\leq_{\I} R_{2}$. We may
assume that $X$ is not finite, otherwise the claim is trivial. From Fact
\ref{fct:restrictionI}, we have that $X\equiv_{\I} X_1\oplus X_2$ where
$X_1\leq R_1$, and $X_2\leq R_2$. By the $\I$-minimality of $R_1$, we have
that either $R_1\leq_{\I} X_1$ or $X_1$ is finite. Similarly, either
$R_2\leq_{\I} X_2$ or $X_2$ is finite. If both $R_1\leq_{\I} X_1$ and
$R_2\leq_{\I} X_2$, then $R_1\oplus R_2\leq_{\I} X$ contradicting
$X<_{\I}R_1\oplus R_2$. Thus without loss of generality we may assume $X_2$
is finite. Then $X\equiv_\I X_1\oplus X_2\leq_{\I}X_1\leq_{\I} R_1$.
\end{proof}

\begin{lemma}\label{secondISMC}
If $\boldsymbol{r}_1$ and $\boldsymbol{r}_2$ are distinct dark minimal
$\I$-degrees represented by $R_1$ and $R_2$, $x$ is even and $y$ is odd, then
the $\mathcal{I}$-degree of $(R_1\oplus R_2)_{/(x,y)}$ is an \ismc of the
pair $\boldsymbol{r}_1$ and $\boldsymbol{r}_2$.
\end{lemma}

\begin{proof}
Suppose $X<_{\I} (R_1\oplus R_2)_{/(x,y)}$: we want to show that either
$X\leq _{\I} R_{1}$ or $X\leq_{\I} R_{2}$. We may assume that $X$ is not
finite, otherwise the claim is trivial. Take $n$ so that $X \leq (R_{1}
\oplus R_{2})_{/(x,y)}\oplus \Id_n$. By Fact \ref{fct:restriction}, we can
write $X\equiv X_0\oplus X_1$ where $X_0\leq (R_{1} \oplus R_{2})_{/(x,y)}$
and $X_1\leq \Id_n$. Thus $X\equiv_{\I} X_0$, and since we are only
considering $X$ up to $\I$-degree, we may replace $X$ by $X_0$ and thus we
may assume that $X< (R_{1} \oplus R_{2})_{/(x,y)}$. Let $W=\range(f)$ where
$f$ reduces $X\leq (R_1 \oplus R_2)_{/(x,y)}$. If $W\cap 2\omega$ hits
infinitely many classes, then $\frac{W}{2}$ hits cofinitely many classes in
$R_1$. Similarly on the odd classes.

We first rule out the possibility that $W$ hits both infinitely many even and
infinitely many odd classes: if this were the case, then  there would be
$h,k$ such that $\frac{W}{2}$ would omit $h$ $R_{1}$-classes and
$\frac{W-1}{2}$ would omit $k$ $R_{2}$-classes. If $f$ does not hit the
$(R_1\oplus R_2)_{/(x,y)}$-class of $x$ then $f$ omits exactly $h+k-1$
classes, otherwise  $f$ omits exactly $h+k$ classes. Let $j$ be the number of
classes (either $h+k-1$ or $h+k$) omitted by $f$. Then by
Fact~\ref{fact:trivial}(2), $X \oplus \Id_{j} \equiv (R_1\oplus
R_2)_{/(x,y)}$, so $X\equiv_{\I} (R_1\oplus R_2)_{/(x,y)}$. This contradicts
the assumption that $X<_{\I} (R_1\oplus R_2)_{/(x,y)}$.

Therefore, we may suppose that $W$ hits only finitely many even classes, the
case of $W$ hitting only finitely many odd classes being similar. Now we aim
to show that $X \leq_{\I} R_{2}$. Let $k$ be the number of the finitely many
even classes that are hit by $f$, and choose representatives $a_{0}, \ldots,
a_{k-1}$ for these classes. Recall that $0, 1, \ldots, k-1$ are
representatives of the $k$ equivalence classes of $\Id_k$. Consider the
function $g$ computed by the following procedure: Given a number $x$, if
$f(x)$ is even then search for the first $a_{i}$ such that $f(x) \rel{(R_{1}
\oplus R_{2})_{/(x,y)}} a_{i}$ and let $g(x)=a_{i}$. If $f(x)$ is odd, then
let $g(x)=f(x)$. Using $g$ we see that if $W$ does not hit the $(R_{1} \oplus
R_{2})_{/(x,y)}$-equivalence class of $x$ then there is a reduction $X\leq
\Id_k\oplus R_2$, which, for each $i <k$, matches e.g.\ the $(R_{1} \oplus
R_{2})_{/(x,y)}$-equivalence class of $a_i$ with the $\Id_k\oplus
R_2$-equivalence class of $2i$. Otherwise there is a reduction $X \leq
(\Id_k\oplus R_2)_{/(0,y)}\equiv \Id_{k-1}\oplus R_2$, where we assume
without loss of generality that the $(R_{1} \oplus
R_{2})_{/(x,y)}$-equivalence classes of $a_0$ and $x$ coincide. In either case, we have $X\leq_{\I}R_2$.
\end{proof}

\begin{lemma}\label{lem:isc-incomparable}
If $R_1$ and $R_2$ are $\I$-incomparable and are each dark minimal or
$\mathbb{Z}$-dark minimal ceers, and $[x_{0}]_{R_1}$ and $[y_{0}]_{R_2}$ are
non-computable, then, letting $x=2x_{0}$ and $y=2y_{0}+1$,we have that
$R_1\oplus R_2$ and $(R_1\oplus R_2)_{/(x,y)}$ are $\I$-incomparable ceers.
(Note that the condition that $[x]_{R_1}$ and $[y]_{R_2}$ are non-computable
must hold if $R_1$ and $R_2$ are each dark minimal, by
Observation~\ref{inseparabilityDarkMinimal}.)
\end{lemma}	

\begin{proof}
Since any ceer $<_{\I}$ either $R_1\oplus R_2$ or $(R_1\oplus R_2)_{/(x,y)}$
must be $\leq_\I$ either $R_1$ or $R_2$ (from the previous two lemmas), we
only need to show that $R_1\oplus R_2$ and $(R_1\oplus R_2)_{/(x,y)}$ are not
$\I$-equivalent ceers, and thus it is enough to show that $(R_1\oplus
R_2)_{/(x,y)} \nleq_{\I} R_1\oplus R_2$. Suppose towards a contradiction that
$(R_1\oplus R_2)_{/(x,y)} \leq_{\I} R_1\oplus R_2$. Then let $n$ be so that
$(R_1\oplus R_2)_{/(x,y)} \leq R_1\oplus R_2\oplus \Id_n$. Consider the composed
reduction $R_1 \leq (R_1\oplus R_2)_{/(x,y)} \leq R_1\oplus R_2\oplus \Id_n$.
In this reduction, let $W$ be the set of elements sent into the copy of $R_1$
in $R_1\oplus R_2\oplus \Id_n$. If $W$ intersects only finitely many
$R_1$-classes, then $R_1\leq_\I R_2$, which contradicts $R_1$ and $R_2$ being
$\I$-incomparable. So, $W$ intersects infinitely many $R_1$-classes, thus it
intersects co-finitely many $R_1$-classes and misses only computable classes.
Thus, $x_{0}$ is sent into $R_1$ in $R_1\oplus R_2$. Similarly, $y_{0}$ is
sent into $R_2$ in $R_1\oplus R_2$ in the
composed reduction $R_2 \leq (R_1\oplus
R_2)_{/(x,y)} \leq R_1\oplus R_2\oplus \Id_n$. Thus the reduction $(R_1\oplus
R_2)_{/(x,y)}\leq R_1\oplus R_2\oplus \Id_n$ sends $x$ into $R_1$ and $y$ into $R_2$. Thus $x$ and $y$ are equivalent in $(R_1\oplus
R_2)_{/(x,y)}$ and their images are inequivalent in $R_1\oplus R_2\oplus \Id_n$, which is a contradiction.
\end{proof}

The previous Lemmas gives us two ways to build $\I$-strongly minimal covers of pairs of $\I$-incomparable minimal dark $\I$-degrees. This assumed that the degrees were represented by ceers which had a non-computable class. We now see that this is a necessary condition for constructing two incomparable $\I$-strongly minimal covers.

\begin{lemma}\label{joinsforZdarkmins}
	Let $\boldsymbol{r}$ be the $\I$-degree of a $\mathbb{Z}$-dark minimal ceer $R$ all of whose classes are computable. Let $\boldsymbol{s}$ be an incomparable $\I$-degree. Then $\boldsymbol{r}\oplus \boldsymbol{s}$ is the join of $\boldsymbol{r}$ and $\boldsymbol{s}$ in the $\I$-degrees.
\end{lemma}
\begin{proof}
	Let $S$ be any ceer. Let $X$ be any ceer so that $R,S\leq_{\I} X$. We want to show that either $R\oplus S \leq_{\I} X$ or $R\leq_{\I} S$. Fix some $n$ so that $R,S\leq X\oplus \Id_n$ and let $f$ and $g$ give these reductions. Let $W$ be the set of pairs $(2x,2y+1)$ so that $f(x) X\oplus \Id_n g(y)$. Then $R\oplus S/W\leq X\oplus \Id_n$. We have two cases to consider. Let $Z=\{x\mid \exists y (2x,2y+1)\in W\}$. Note that $Z$ is $R$-closed.
	
	Case 1: $Z$ contains only finitely many $R$-classes. In this case,
since each of these classes are computable, $Z$ is computable. Let $R_0= R \restriction \omega\smallsetminus Z$. Then $R\oplus S \equiv_{\I} R_0\oplus S\leq X\oplus \Id_n$, so we conclude $R\oplus S\leq_{\I} X$.
	
	Case 2: $Z$ contains infinitely many $R$-classes. In this case, we know that $Z$ contains co-finitely many $R$-classes. Since each of the remaining classes are computable, $Z$ is computable.
Let $R_0= R \restriction Z$. Note that $R_0\equiv_{\I} R$. We now give a map which reduces $R_0$ to $S$. For $x\in R_0$, search for some $y$ so that $(2x,2y+1)\in W$. Send $x$ to the first such $y$ found. This gives a reduction of $R_0$ to $S$, thus $R\leq_{\I} S$.
\end{proof}

\begin{lemma}\label{ZsAreIsolated}
If $A$ is a $\mathbb{Z}$-dark minimal ceer and $B$ is any ceer so that the
pair $A,B$ has two $\I$-incomparable {\ismc}s, then $A$ has a non-computable
class.
\end{lemma}

\begin{proof}
Immediate from the preceding lemma.
\end{proof}

\section{Coding graphs into the partial order $\Dark_{/\I}$ using
parameters} \label{Coding Using Parameters} In the following by a \emph{graph} we mean a structure $G=\langle
V, E\rangle$ where $V$ is a nonempty set of \emph{vertices}, and $E$, called
the \emph{edge relation},  is an  irreflexive and symmetric binary
relation on $V$.

\begin{defn}\label{def:vertices-edges}
Given any dark $\I$-degree $\boldsymbol{c}$, we describe a graph
$G_{\boldsymbol{c}}$ as follows:
\begin{itemize}
  \item (vertices) the \emph{vertex set} $V$ of the graph are the
      $\I$-minimal degrees $\leq \boldsymbol{c}$.
  \item (edges) For each pair $\boldsymbol{d}, \boldsymbol{e} \in V$, we
      put an \emph{edge} between $\boldsymbol{d}$ and $\boldsymbol{e}$ if
      and only if there are two incomparable $\I$-degrees $\boldsymbol{a},
      \boldsymbol{b} \leq_{\I} \boldsymbol{c}$ which are both {\ismc}s of
      the pair $\boldsymbol{d}, \boldsymbol{e}$.
\end{itemize}
\end{defn}
We will often say that we have an edge between ceers $R$ and $S$ if we have an edge between their $\I$-degrees.

We now show how to code any computable graph  in $\Dark_{/\I}$, modulo some isolated vertices.

\begin{thm}\label{codesExist}
If $G=(V,E)$ is a computable graph, then there exists a dark degree
$\boldsymbol{c}$ so that $G_{\boldsymbol{c}}$ is isomorphic to the disjoint union of $G$ with
a graph which has no edges.
\end{thm}
\begin{proof}
Fix a computable presentation of $G$. We also fix a uniform c.e. sequence of
pairwise $\leq_\I$-incomparable dark minimal ceers $\{R_i\mid i\in \omega\}$:
for this, just observe that by the proof of \cite[Theorem~3.3]{Andrews-Sorbi}
from any finite set $R_0, \ldots, R_n$ of dark minimal ceers we can uniformly
find a dark minimal ceer $R_{n+1}$ so that $R_{n+1} \nleq R$ where
$R=\bigoplus_{i\leq n} R_i$, and thus $R_{n+1} \nleq_\I R$ by
Observation~\ref{inseparabilityDarkMinimal}; at the same time, for each $i$,
$R_i \nleq_\I R_{n+1}$ otherwise (again by
Observation~\ref{inseparabilityDarkMinimal}) $R_i \leq R_{n+1}$, and thus
$R_{n+1} \leq R_i \leq R$, by minimality of $R_{n+1}$. If $R,E$ are ceers we
say that the \emph{$n$-th column of $E$ codes $R$} (or $R$ is \emph{copied in
the $n$-th column of $E$}) if for every $x,y$, $x \rel{R} y$ if and only if
$\langle n,x \rangle \rel{E} \langle n,y \rangle$.

\subsection*{Requirements and strategies}
We construct the ceer $C$ (with $\I$-degree $\boldsymbol{c}$) with the
following requirements:

\begin{flalign*}
\Code_i: &\quad \textrm{Some column of $C$ codes $R_i$.}\\
\dark_j: &\quad \textrm{If $W_j$ is infinite, then there are distinct $x, y\in
    W_j$ so that $x \rel{C} y$.}\\
\Edge_{\langle i, j \rangle}: &\quad \textrm{If there is an edge between $i$ and
    $j$ in $G$, then some column}\\
    &\quad \textrm{of $C$ codes $(R_i\oplus R_j)_{/(0,1)}$.}
\end{flalign*}

We have no requirement directly ensuring that there are no extra edges in $G_{\boldsymbol{c}}$. This will follow from the simple form of the ceer $C$ that we construct.

The \emph{priority order} of the requirements is
\[
\Code_{0} < \Edge_{0} < \dark_{0}< \cdots < \Code_{j} <
\Edge_{j} < \dark_{j} < \cdots
\]
A requirement $\Edge_{\langle i, j \rangle}$ such that the graph $G$ has an
edge between $i$ and $j$ will be called \emph{binding}. As we consider only
symmetric irreflexive graphs, we assume that if $(i)_{0}=(i)_{1}$ then
$\Edge_{i}$ is not binding, and $\Edge_{\langle j, i \rangle}$ is not binding
if $\langle i, j \rangle < \langle j,i \rangle$.

We outline the strategies to meet the requirements.

For the sake of the $\Code_i$-requirement, we act by picking a new column and
determining that this column will copy $R_i$. It restrains this entire
column.

For the sake of the $\dark_j$-requirement, while the requirement is not
satisfied (it becomes permanently satisfied when distinct numbers $x,y$
appear such that $x,y \in W_{j}$ and $x \rel{C} y$) we simply wait for $W_j$
to enumerate two distinct numbers $x,y$ which are not in  columns restrained
by higher priority requirements, then we collapse to a single class the
entire columns of $x$ and $y$.

We will use the following notations: $\omega^{[i]}$ denotes the $i$-th column
of $\omega$; $\omega^{[i,j]}= \bigcup_{i \leq r \leq j}\omega^{[r]}$, and
$\omega^{[i,j)}= \bigcup_{i \leq r < j}\omega^{[r]}$. It will follow from the
construction that $\dark_j$ works with a parameter $\epsilon_j$ so that the
columns restrained by higher priority requirements will be the columns
$\omega^{[n]}$ for $n \leq \epsilon_{j}$; eventually $\epsilon_{j}$
stabilizes in the limit, and the interval $[0, \epsilon_{j}]$ is eventually
partitioned into subintervals, each one being either a singleton $\{i\}$ so
that in the $i$-th column codes a dark ceer (with $R_{0}$ coded in the $0$-th
column); or a subinterval $[a,b]$ so that all columns $\omega^{[i]}$ with $i
\in [a,b]$ are collapsed to a single (clearly decidable) class. Being thus
computably bijective with a uniform join of dark ceers and copies of
$\Id_{1}$ and being infinite, this finite set of columns can be viewed as a
dark ceer $R$ (see Fact~\ref{fct:dark-closure}). If an infinite $W_j$ is
contained in this finite set of columns, we need not do anything, as in this
case $W_{j}$ is not a transversal of $C$. Otherwise (as shown in the
verification) from $W_{j}$ one could find an infinite c.e. set which is a
transversal
contained in one column of $C$, which is impossible since each column of $C$ is either finite or codes a dark ceer. In the following
we will distinguish between \emph{coding columns} in which we code dark
ceers, and \emph{column blocks}, comprised of finitely many consecutive
columns of $\omega$ all collapsed to a single $C$-equivalence class which is
decidable.

If $\dark_{j}$ acts by collapsing, it re-initializes all lower priority
requirements, it leaves untouched all coding columns and column blocks in the
restrained interval $[0, \epsilon_{j}]$, and collapses to a single class all
other columns up to the biggest column so far used in the construction: these
newly collapsed columns contain also the witnesses $x,y$ which are thus
collapsed and $\dark_j$ is permanently satisfied.

For the sake of the $\Edge_{\langle i, j \rangle} $-requirement, we act
exactly as in the other Coding requirements $\Code_i$. The only distinction
is that the existence of any one of these $\Edge$-requirements (i.e. whether
or not it is binding) is determined by the computable graph $G$.

It follows that in the end $C$ will consist of single coding
columns (used for coding ceers of the form $R_i$ or $(R_i\oplus
R_j)_{/(0,1)}$) and column blocks comprised of finitely many consecutive
columns of $\omega$ all collapsed to a single $C$-equivalence class. We say that a ceer with this structure is a \emph{simple coding ceer}.

We first observe that the only dark minimal ceers $\leq_{\I} C$ are the
$R_i$'s that we began with. To see this, let $E\leq_{\I}C$ be dark minimal. Firstly notice that $E \leq C$ by
Observation~\ref{inseparabilityDarkMinimal}, and $E$ must reduce to a single coding
column. In fact  (again by Observation~\ref{inseparabilityDarkMinimal}) no
class can be mapped to a column block as no $E$-class is decidable, and no
two distinct $E$-classes can be mapped to distinct coding columns by
computable inseparability of $E$. So $E \leq R_i$, or $E \leq (R_i\oplus
R_j)_{/(0,1)}$ for some $i,j$. In the former case $E$ is equivalent to $R_i$, and
in the latter case, Lemma \ref{secondISMC} shows that $E$ must be equivalent
to either $R_i$ or $R_j$.

\subsection*{Adequacy of the requirements}
We now suppose that $C$ is a simple coding ceer and has been constructed satisfying all the requirements
and we verify that $G_{\boldsymbol{c}} \simeq G$ (where $\simeq$ denotes
isomorphism) modulo a set of isolated vertices, which are $\I$-degrees
of $\mathbb{Z}$-dark minimal ceers with only computable classes.

We first observe that the only $\mathbb{Z}$-dark minimal ceers $\leq_{\I} C$
have only computable classes. Suppose towards a contradiction that $A$ is a
$\mathbb{Z}$-dark minimal ceer with a non-computable class $[m]_A$, and $f$ is a reduction $f: A\leq C
\oplus \Id_r$ for some $r$ which witnesses $A\leq_{\I} C$. Let $f(m)=2\cdot\langle n,x\rangle$, where
clearly the $n$-th column is a coding column, say coding the ceer $X$, by
undecidability of $[m]_A$. Let $W$ be the computable set of $y$ so that
$f(y)$ does not land in the $n$-th column of $C$. By
Corollary~\ref{cor:Z-dark-infy-computable} $W$ cannot hit infinitely many
classes, as otherwise $[m]_A$, which is omitted by $W$, should be computable.
Thus we must have that $W$ intersects only finitely many classes. Thus the
reduction $f$ lands in the $n$-th column of $C$ plus a finite collections of
single equivalence classes of $C\oplus \Id_r$. Thus $A\leq X\oplus \Id_s$ for
some $s$, so $A\leq_{\I}X$. Now, $X$ is either a dark minimal ceer or one of
$(R_i\oplus R_j)_{/(0,1)}$. Since a dark minimal ceer cannot bound a
$\mathbb{Z}$-dark minimal ceer, the former case is impossible. In the latter
case, we have that $A\equiv_{\I} (R_i\oplus R_j)_{/(0,1)}$ or $A\leq_{\I}
R_i$ or $A\leq_{\I}R_j$. In any case $A$ is not a $\mathbb{Z}$-dark minimal
ceer, which is a contradiction.

Thus, the universe of $G_{\boldsymbol{c}}$ is comprised of the dark minimal
ceers $R_i$ along with perhaps some $\mathbb{Z}$-dark minimal ceers which
have all computable classes. By Lemma \ref{ZsAreIsolated}, these are isolated
points in $G_{\boldsymbol{c}}$.

If there is an edge between $i$ and $j$ in $G$, then we have $R_i\oplus R_j$
and $(R_i\oplus R_j)_{/(0,1)}$ being both $\leq C$. By Lemmas \ref{oplusISMC}
and \ref{secondISMC}, we have an edge between  $R_i$ and $R_j$ in
$G_{\boldsymbol{c}}$.

Now, suppose that there is no edge between $i$ and $j$ in $G$. Then we do not
place any columns in $C$ of the form of $(R_i\oplus R_j)_{/(0,1)}$. Suppose
$X$ is an \ismc below $C$ of the pair $R_i,R_j$. Consider the composed
reductions $R_i\leq_\I X \leq_\I C$ and $R_j\leq_\I X\leq_\I C$.
By
Observation~\ref{inseparabilityDarkMinimal}, the first reduction in each of
the two chains is $\leq$. By computable inseparability of the classes of $R_i$ and $R_j$, these
reductions reduce to single coding columns of $C$. Each coding column is
either some $R_k$, or has the form $(R_k\oplus R_l)_{/(0,1)}$. In the latter
case, by Lemma~\ref{secondISMC} only $R_k$ and $R_l$ reduce to that column.
Thus, the two coding columns in which $R_i$ and $R_j$ are reducing to in $C$ are
different columns. Thus we see $R_i\oplus R_j\leq X$, giving $R_i\oplus
R_j\equiv_\I  X$ as $X$  is an \ismc of the pair $R_i,R_j$. Thus we can only
have one \ismc of the pair $R_i,R_j$ below $C$, and there is no edge between
$R_i$ and $R_j$ in $G_{\boldsymbol{c}}$.

\subsection*{The construction}
In the formal construction we make use of several parameters: $\gamma_i(s)$,
if defined, denotes the column in which at $s$ we code $R_i$.
$\epsilon_{\langle i,j \rangle}(s)$, if defined and $\Edge_{\langle i,j
\rangle}$ is binding, denotes the column in which at $s$ we code $(R_i \oplus
R_j)_{/(0,1)}$. The parameter $r(s)$ denotes the least number $n$ so that the
corresponding column is still \emph{fresh} i.e. no parameter $\gamma_i(t)$ or
$\epsilon_i(t)$ for $t\leq s$ was defined and $\geq n$. At each stage $s$
there will always be a unique number $i$ such that we define for the first
time, or redefine, $\gamma_{i}(s)$. This will determine also the definition
of $\epsilon_{i}(s)$ as
\[
\epsilon_{i}(s)=
\begin{cases}
\gamma_{i}(s)+1, &\text{if $\Edge_{i}$ is binding},\\
\gamma_{i}(s), &\text{otherwise}.
\end{cases}
\]
This means that we plan to code $R_{i}$ in the $\gamma_{i}(s)$-column. If
there is an edge in the graph from $(i)_{0}$ to $(i)_{1}$ then we plan to
code $(R_{(i)_{0}} \oplus R_{(i)_{0}})_{/(0,1)}$ in the next column, and
otherwise $\Edge_{i}$ is not binding and does not need to be coded in any
column. A requirement $\dark_{j}$ is \emph{satisfied} at $s$, if $W_{j}$ has
already enumerated a pair of distinct numbers $x,y$ which $C$ has already
collapsed. Finally, at each stage $s$ we define a ceer $C_{s}$, so that
$C_{0}\subseteq C_{1}\subseteq \cdots$, and the sequence $\{C_{s}\mid  s \in
\omega\}$ is c.e., so that $C=\bigcup_{s} C_{s}$ is our desired final ceer.

\paragraph{Stage $0$}
Let $\gamma_0(0)=0$. Since $0= \langle 0, 0 \rangle$, $\Edge_0$ is not
binding, thus $\epsilon_{0}(0)=0$, and $r(0)=1$. All other parameters are
undefined. (The construction will ensure that $\gamma_0$ and $\epsilon_0$
will never be initialized.) Let $C_0$ be the ceer generated by the c.e. set
of pairs $X$ where
\[
X= \{\gamma_0(0)\}\times R_0(=\{0\}\times R_0)
\]

Here and below we use the notation: For $n\in \omega$ and $E$ a ceer,
$\{n\}\times E = \{(\langle n,x\rangle,\langle n, y \rangle )\mid (x,y)\in
E\}$.

\paragraph{Stage $s+1$}
Let us say that $\dark_j$ \emph{requires attention at $s+1$} if $j \leq s$,
$\epsilon_j=\epsilon_j(s)$ is defined, $\dark_j$ is not as yet satisfied at
the end of stage $s$, and there are distinct numbers $x,y \in W_{j,s+1}\cap \omega^{[\epsilon_j+1, r(s))}$.

\begin{enumerate}

\item If some $\dark_j$ requires attention, then pick the least such $j$ and define $\gamma_{j+1}(s+1)=r(s)$. This also determines $\epsilon_{j+1}(s+1)$ (equal to $r(s)+1$ or $r(s)$ depending on whether
    $\Edge_{j+1}$ is binding or not). Set to be undefined all
    $\gamma_i$ and $\epsilon_i$ for all $i>j+1$. Let $C_{s+1}$ be the ceer generated by the c.e. set of
    pairs $C_s \cup X$ where, for the newly defined $\gamma_{j+1},
    \epsilon_{j+1}$,
    \[
      X=
        \begin{cases}
              (\omega^{[\epsilon_{j+1},r(s))})^2 \cup
              \left(\{\gamma_{j+1}\}\times R_{j+1}\right) \cup
              \left(\{\epsilon_{j+1}\}
              \times (R_{(j+1)_0} \oplus R_{(j+1)_1})_{/(0,1)}\right),
                   &\text{if $\Edge_{j+1}$}\\
                   &\text{is binding},\\
              (\omega^{[\epsilon_{j+1},r(s))})^2 \cup \left(\{\gamma_{j+1}\}
              \times R_{j+1}\right),
                   &\text{otherwise},
\end{cases}
\]
    Notice that the pairs in $\omega^{[\epsilon_{j+1},r(s))}$ all
    $C$-collapse. Declare $\dark_j$ \emph{satisfied} (it will never become
    unsatisfied again), as $\dark_j$ has $C$-collapsed two distinct numbers
    of $W_j$. Notice that $\dark_j$ injures all lower priority $\Code$- and
    $\Edge$-requirements. The injured highest priority $\Code$-requirement
    (namely, $\Code_{j+1}$) starts anew on the fresh column $r(s)$,  and,
    if binding, $\Edge_{j+1}$ starts anew on the next column.

\item If no $\dark_j$ requires attention then let $i$ be the least number
    such that $\gamma_{i}(s)$ is undefined. Define $\gamma_{i}(s+1)=r(s)$.
    This determines  $\epsilon_{i}(s+1)$ and $r(s+1)$ as well. Let $C_{s+1}$ be the ceer generated by the c.e. set of pairs $C_s \cup X$
    where, for the newly defined $\gamma_{i}= \gamma_{i}(s+1)$ and
    $\epsilon_{i}= \epsilon_{i}(s+1)$,
\[
X=
\begin{cases}
 C_s \cup
      \left(\{\gamma_{i}\}\times R_{i}\right) \cup \left(\{\epsilon_{i}\}
      \times (R_{(i)_0} \oplus R_{(i)_1})_{/(0,1)}\right),
      &\text{if $\Edge_{i}$ is binding},\\
C_s  \cup \left( \{\gamma_{i}\}\times R_{i}\right), &\text{otherwise}.
\end{cases}
\]
\end{enumerate}

\subsection*{Verification}
We first observe that for every $i$, $\gamma_{i}=\lim_{s} \gamma_{i}(s)$ and
$\epsilon_{i}= \lim_{s} \epsilon_{i}(s)$ exist. By the way we define
$\epsilon_{i}$ we need only show that $\lim_{s} \gamma_{i}(s)$
exists, as $\epsilon_{i}=\gamma_{i}+1$ if $\Edge_{\langle
(i)_{0}(i)_{1} \rangle}$ is binding, and $\epsilon_{i}=\gamma_{i}$ otherwise.
This is seen by induction. The claim is trivial if $i=0$ as $\gamma_{0}(s)=0$ for every
$s$.

Suppose that the claim is true of every $j\leq i$, and let $s_{0}$ be the
least stage such that for every $s\ge s_{0}$ no such $\gamma_{j}$  changes at
$s$. Consider the requirement $\dark_{i}$. If it has already been satisfied
by stage $s_{0}$, or it will never act, then for every $s \ge s_{0}$
$\gamma_{i+1}(s)=\epsilon_{i}+1$. On the other hand, if at some least stage $s_{1}
\ge s_{0}$ $\dark_{i}$ acts, it becomes satisfied, by collapsing to a single
class the column block $\omega^{[\epsilon_{i}+1, r(s_1))}$. Then
$\gamma_{i+1}(s)=r(s_1)$ for every $s \ge s_{1}$.

Finally we prove that $C$ is dark, by showing that each $\dark_i$ is
satisfied. Suppose that $W_{i}$ is infinite and let $s_{0}$ be a stage such
that $\gamma_{i}$ and $\epsilon_{i}$ never change after $s_{0}$. Thus the
columns $\omega^{[j]}$ with $j \leq \epsilon_{i}$ are partitioned in coding
columns, and column blocks, and the intersection of $C$ with these columns
will never change after $s_{0}$. Call $E$ this intersection. Clearly there is
a computable bijection $f$ of $\omega^{[0,\epsilon_{i}]}$ onto $\omega$ under
which $E$ is translated into a ceer of the form $E'=E_{1}\oplus \cdots \oplus
E_{m}$ where each $E_{k}$ is $\equiv$ to either some $R_{j}$ or $(R_{j}
\oplus R_{h})_{/(0,1)}$ for some $j,h$, or  a single class. Since the $0$-th
column codes $R_{0}$ and thus $E'$ is not finite, it follows that $E'$ is
dark. By the same computable function $f$, $W_{i}\cap
\omega^{[0,\epsilon_{i}]}$ is transformed into a c.e. set $W=f[W_{i}]$. If
$W_{i}$ were an infinite transversal of $C$ and did not contain infinitely
many elements in the complement of $\omega^{[0,\epsilon_{i}]}$, then $W$
minus a finite set would be an infinite transversal of $E'$, which would
contradict the darkness of $E'$. Thus if $W_i$ is an infinite transversal of
$C$, then there are distinct $x,y \in W_{i}$ such that $x,y \notin
\omega^{[0,\epsilon_{i}]}$. But then $\dark_{i}$ would act for the sake of some such
pair $x,y$, and thus $\dark_{i}$ would collapse such a pair $x,y$. We conclude that $W_{i}$ is not a transversal of $C$.
\end{proof}

\begin{remark}\label{rem:xyundecidable}
Notice that instead of choosing $0,1$ when adding $(R_i \oplus R_j)_{/(0,1)}$
to $C$ we could have chosen any pair $(x,y)$ with $x$ even and $y$ odd, and
added $(R_i \oplus R_j)_{/(x,y)}$. The argument in the previous proof relies
on the fact that we can apply Lemma~\ref{lem:isc-incomparable} which works as
long as the equivalence classes of $x$ and $y$ are not computable. Since $R_i$
and $R_j$ are dark minimal, Observation~\ref{inseparabilityDarkMinimal} shows that all of their equivalence classes are
non-computable.
\end{remark}

\begin{remark}\label{rem:finite-infinite}
In the proof of Theorem~\ref{codesExist} we start with a uniform c.e.
sequence of $\leq_{\I}$-incomparable dark minimal ceers $\{R_{i} \mid i \in
\omega\}$, as we are tacitly assuming that we need to code a graph with
infinitely many vertices. If we need to code a finite graph, we can simply
code a finite sequence of incomparable dark minimal ceers $R_i$ along with
ceers $(R_i\oplus R_j)_{/(0,1)}$ to code edges between $i$ and $j$. We
needn't even have any $\dark_i$ requirements. In this case the $\I$-degree
$\boldsymbol{c}$ is the $\I$-degree of the uniform join of finitely many dark
ceers, and thus it is automatically dark.
\end{remark}

\section{Interpreting $(\mathbb{N},+,\times)$ in the partial order
$\Dark_{/\I}$ without parameters}\label{sct:dark-arithmetic}

We are now ready to show how to give a definition of $(\mathbb{N},+,\times)$
without parameters. As described in section \ref{outlineofproof}, our strategy is to consider the codes $\boldsymbol{c}$ so that $G_{\boldsymbol{c}}$ encodes a model of Robinson's system $Q$. Then the key is in encoding all finite functions to define an equivalence relation which picks out which elements of these different models of $Q$ are the same number. By encoding all finite functions, we see that this equivalence relation is at least correct on all standard numbers. Finally, our copy of $\mathbb{N}$ will the be equivalence classes of the standard natural numbers, which we will isolate as the classes that have a representative in every encoded model of $Q$.

\begin{defn}
Let $\mathcal{P}=(P,\leq)$ be a poset, and $n \ge 1$. An relation $R
\subseteq P^n$ is said to be $\emptyset$-\emph{definable in $\mathcal{P}$} if there is a
first order formula without parameters $\phi(\vec{x})$ in the language of posets (with $\vec{x}$
an $n$-tuple of variables, and all free variables of $\phi$ are in
$\vec{x}$) such that, for every $\vec{a}\in P^n$,
\[
R(\vec{a}) \Leftrightarrow \mathcal{P} \models
\phi(\vec{a}).
\]
\end{defn}

\begin{corollary}\label{cor:def-of-vertices-and-edges}
There are $\emptyset$-definable relations $V(x,c), E(x,y,c), NI(x,c)$ on $\Dark_{/\I}$
such that if $\boldsymbol{c}$ is a dark $\I$-degree then
\begin{itemize}
  \item $G_{\boldsymbol{c}}=\{\boldsymbol{x} \mid
      V(\boldsymbol{x},\boldsymbol{c})\}$,
  \item $E(\boldsymbol{x}, \boldsymbol{y}, \boldsymbol{c})$ if and only if
      $\boldsymbol{x}, \boldsymbol{y}\in G_{\boldsymbol{c}}$ and there is
      an edge between $\boldsymbol{x}$ and $\boldsymbol{y}$,
  \item $\{\boldsymbol{x} \mid NI(\boldsymbol{x},\boldsymbol{c})\}$ is the
      set on non-isolated vertices of $G_{\boldsymbol{c}}$.
\end{itemize}
\end{corollary}

\begin{proof}
Immediate as the definitions of vertices and edges in
Definition~\ref{def:vertices-edges} are given in terms of minimality and
existence of strong minimal covers for pairs, which are first order
properties in the language of posets.
\end{proof}

\begin{remark}\label{rem:coding-G-in-DarkI}
From the previous corollary, we see how to effectively translate any sentence
$\sigma$ in the language of graph theory (just the binary edge relation) into
a formula $\tilde{\sigma}(w)$ of posets with free variable $w$ such that for
every dark $\I$-degree $\boldsymbol{c}$,
\[
G_{\boldsymbol{c}} \models \sigma \Leftrightarrow \Dark_{/\I} \models
\tilde{\sigma}(\boldsymbol{c}).
\]
\end{remark}

We will refer to the following result:

\begin{thm}\label{thm:folklore}
There is a computable graph $G$ without isolated vertices in which
$(\mathbb{N}, +, \times)$ is first order $\emptyset$-definable, that is there are
first-order formulas without parameters $U,\phi_+,\phi_\times$  in the language of graphs
defining respectively the subset which is the universe of the copy of
$\mathbb{N}$ and the operations $+, \times$ in $G$.
\end{thm}

\begin{proof}
See item 1(c)14 of the list of def-complete structures of \cite{Korec}, or
see \cite[Theorem~5.5.1]{Hodges}.
\end{proof}

\begin{remark}\label{lem:only-isolated}
In view of Theorem~\ref{thm:folklore} in coding $G$ in $\Dark_{/\I}$  we only
need the subset of the vertices of $G_{\boldsymbol{c}}$
which is comprised of the non-isolated vertices (all of them being dark
minimal). Since by Corollary~\ref{cor:def-of-vertices-and-edges} this
subgraph is recognizable in a first order way from parameter $\boldsymbol{c}$,
henceforth we shall use the symbol $G_{\boldsymbol{c}}$ to denote this
subgraph, so that  $G_{\boldsymbol{c}}$  is henceforth understood to be
without isolated vertices and $G\simeq G_{\boldsymbol{c}}$.
\end{remark}

Fix a graph $G$ as in Theorem~\ref{thm:folklore}. So there are first-order
formulas $U,\phi_+,\phi_\times$  and a mapping $\sigma \mapsto \sigma^\circ$
from arithmetical formulas to formulas in the language of graphs so that
the following hold,
where $:=$ denotes syntactic equality: $(+(x,y,z))^\circ :=\phi_+(x,y,z)$,
$(\times(x,y,z))^{\circ}:=\phi_\times(x,y,z)$, the mapping $\mbox{}^\circ$
commutes (modulo $:=$) with propositional connectives, $(\forall x
\sigma)^\circ:=(\forall x)(U(x) \rightarrow \sigma^\circ)$, $(\exists x
\sigma)^\circ:=(\exists x)(U(x) \wedge \sigma^\circ)$, and finally for every
sentence $\sigma$, $\mathbb{N} \models \sigma$ if and only if $G\models
\sigma^\circ$.

\begin{remark}\label{rem:from-N-to-po}
Using Corollary~\ref{cor:def-of-vertices-and-edges} we see that the following
binary relation $U^c(x)$, and quaternary relations $\phi_+^c(x,y,z),
\phi_\times^c(x,y,z)$ (corresponding to $U(x), \phi_+(x,y,z),
\phi_\times(x,y,z)$ mentioned above) are $\emptyset$-definable in $\Dark_{/\I}$:
\begin{itemize}
  \item $U^{\boldsymbol{c}}(\boldsymbol{x})$ if and only if $G_{\boldsymbol{c}}
  \models U(\boldsymbol{x})$ (henceforth let $U^{\boldsymbol{c}}=
  \{\boldsymbol{x}\mid U^{\boldsymbol{c}}(\boldsymbol{x})\}$),
  \item  $\phi^{\boldsymbol{c}}_+(\boldsymbol{x}, \boldsymbol{y},
      \boldsymbol{z})$ if and only if $G_{\boldsymbol{c}} \models
      \phi_+(\boldsymbol{x}, \boldsymbol{y}, \boldsymbol{z})$,
  \item $\phi^{\boldsymbol{c}}_\times(\boldsymbol{x}, \boldsymbol{y},
      \boldsymbol{z})$ if and only if $G_{\boldsymbol{c}} \models
      \phi_\times(\boldsymbol{x}, \boldsymbol{y}, \boldsymbol{z})$.
\end{itemize}
For every $\boldsymbol{c} \in \Dark_{/\I}$ we can regard the triple
$(U^{\boldsymbol{c}}, \phi_+^{\boldsymbol{c}}, \phi_\times^{\boldsymbol{c}})$
as a structure for the arithmetical language $+, \times$. From these, we also
have $\emptyset$-definable relations $\phi^{c}_S(x, y), \leq^{c}(x, y),
\boldsymbol{0}^{c}(x)$ in $\Dark_{/\I}$, corresponding to the formulas
defining in $G$ the successor operation, the natural ordering on
$\mathbb{N}$, and the number $0$, respectively.
\end{remark}

\begin{defn}\label{def:good-code}
A $\boldsymbol{c} \in \Dark_{/\I}$ is a \emph{good code} if, in
$G_{\boldsymbol{c}}$, $(U^{\boldsymbol{c}}, \phi_+^{\boldsymbol{c}},
\phi_\times^{\boldsymbol{c}})$ gives a model of Robinson's system $Q$. 	
\end{defn}

\begin{corollary}\label{cor:good-definable}
The set of good codes is $\emptyset$-definable in the $\Dark_{/\I}$ degrees.
\end{corollary}

\begin{proof}
This immediately follows from the fact that $Q$ is finitely axiomatizable.
\end{proof}

\begin{remark}\label{rem:operations}
If $\boldsymbol{c}$ is a good code then  $\phi_S^{\boldsymbol{c}},
\phi_+^{\boldsymbol{c}}, \phi_\times^{\boldsymbol{c}}$ define operations
in $U^{\boldsymbol{c}}$, and $\boldsymbol{0}^{\boldsymbol{c}}$ is a distinguished
element of $U^{\boldsymbol{c}}$.
\end{remark}

We now begin encoding functions between our different models $U^{\boldsymbol{c}}$ and $U^{\boldsymbol{c'}}$. The goal of these functions is to def

\begin{defn}
For any pair of dark minimal $\I$-degrees $(\boldsymbol{x}, \boldsymbol{y})$,
we say that a graph of the form $\boldsymbol{x} \rel{E} \boldsymbol{a}
\rel{E} \boldsymbol{d} \rel{E} \boldsymbol{y}$ and $\boldsymbol{a} \rel{E}
\boldsymbol{b} \rel{E} \boldsymbol{c} \rel{E} \boldsymbol{a}$ (where
$\boldsymbol{a}, \boldsymbol{b}, \boldsymbol{c}, \boldsymbol{d}$ are
distinct, and distinct from $\boldsymbol{x}, \boldsymbol{y}$) is a
\emph{graph-label for the pair $(\boldsymbol{x}, \boldsymbol{y})$}. See
Figure~\ref{fig:figure1}.
\end{defn}

\begin{figure}[h!]
  \centering
  \includegraphics[scale=.5]{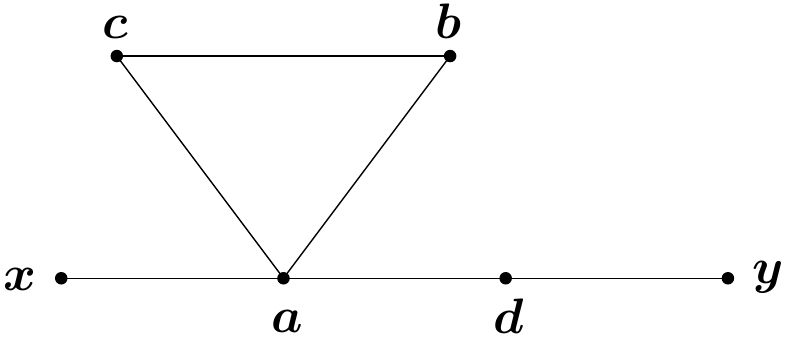}
  \caption{A graph-label for the pair $(\boldsymbol{x},
\boldsymbol{y})$}\label{fig:figure1}
\end{figure}

\begin{defn}\label{defn:finite-functions}
Given a set $F=\{(\boldsymbol{x}_i, \boldsymbol{y}_i)\mid i < n\}$ of pairs
of dark minimal $\I$-degrees,
we
say that an $\boldsymbol{f}\in \Dark_{/\I}$ \emph{is a name for $F$} if
$G_{\boldsymbol{f}}$ is comprised of the union of graph-labels for the pairs
$(\boldsymbol{x}_i,\boldsymbol{y}_i)$, where the various quadruples
$\boldsymbol{a}, \boldsymbol{b}, \boldsymbol{c}, \boldsymbol{d}$ are distinct
for each pair, and not appearing in $F$.
\end{defn}

\begin{lemma}\label{namesExist}
If $F$ is a finite set of pairs of ceers, each of which is either dark
minimal or is $\mathbb{Z}$-dark minimal with a non-computable class, then
there is a dark $\I$-degree which is a name for the set of pairs of
$\I$-degrees corresponding to the pairs of ceers in $F$.
\end{lemma}

\begin{proof}
Fix a pair $(X,Y)$ in $F$. Without loss of generality, we assume that $[0]_X$
and $[0]_Y$ are non-computable. Let $A,B,C,D$ be distinct dark minimal ceers
with $\I$-degrees not mentioned in $F$, and chosen only for this pair. Then
we construct
\begin{multline*}
Z_{(X,Y)}=X \oplus A \oplus D \oplus Y \oplus B \oplus C \\
\oplus (X\oplus A)_{/(0,1)})
\oplus ((A\oplus D)_{/(0,1))}) \oplus
((D\oplus Y)_{/(0,1))})\oplus \\
((A\oplus B)_{/(0,1))}) \oplus
((B\oplus C)_{/(0,1))}) \oplus
(A\oplus C)_{/(0,1))}).
\end{multline*}
Now that we have constructed $Z_{(X,Y)}$ for each pair $(X,Y)\in F$, define
\[
f=\bigoplus_{(X,Y)\in F}Z_{(X,Y)}.
\]
Since there are only finitely many direct summands in $f$, and each of these
are dark, $f$ is dark by Fact \ref{fct:dark-closure} (2). Let
$\boldsymbol{f}$ denote the $\I$-degree of $f$.

First we check that we have no unwanted vertices, i.e. the only minimal
$\I$-degrees below $\boldsymbol{f}$ are equal to the $\I$-degrees of the
ceers $X,Y,A,B,C,D$ that we placed there. If $R$ has minimal $\I$-degree and
$R\leq_{\I} f$, then by
Fact~\ref{fct:restrictionI} $R\equiv_{\I} \oplus_j E_j$ where each $E_j$ is
$\leq$ one of the summands in the definition of $f$. That is, each $E_j$ is either
$\leq_{\I}$ some dark minimal ceer or $E_j$ is $\leq_{\I} (S\oplus
T)_{/(0,1)}$ where $S$ and $T$ are in dark minimal $\I$-degrees. Thus by
Lemma~\ref{secondISMC}, each $E_j$ is either finite or its $\I$-degree is
$\geq_{\I}$
a dark minimal $\I$-degree, namely the $\I$-degree of one of the summands. Thus since $R$ is not finite, the
$\I$-degree of one of the $E_j$ is $\geq_{\I}$ the dark minimal
$\I$-degree of one of the summands, and since $R$ has minimal $\I$-degree,
$R$ is $\I$-equivalent to one of the summands.

Next we check that $f$ codes exactly the edges we intended. Since the
equivalence class of $0$ is non-computable in all of the ceers $X,Y,A,B,C,D$
that we consider, if we place columns for $X$, $A$, and $(X\oplus
A)_{/(0,1)}$, we have ensured that the $\I$-degrees of $X$ and $A$ have two
$\I$-incomparable $\I$-strongly minimal covers below $\boldsymbol{f}$.
Similarly for the pairs $(A,D)$, $(D,Y)$, $(A,B)$, $(B,C)$, and $(A,C)$. Thus
by Lemmas \ref{oplusISMC} and \ref{secondISMC}, $f$ successfully codes every
edge that we intended. Suppose now the $\I$-degree of $X$ is an $\I$-strongly
minimal cover below $\boldsymbol{f}$ of the $\I$-degrees of  the pair $R_1$
and $R_2$, which are minimal $\I$-degrees below $\boldsymbol{f}$ between
which we did not explicitly code an edge. We may assume that $R_1$ and $R_2$
are among the summands we used to create $f$, in particular the equivalence
class of $0$ is non-computable. Then consider the pair of reductions $R_1
\leq f \oplus \Id_n$ and $R_2  \leq f\oplus \Id_n$, which we get by composing
the reductions $R_1\leq X \oplus \Id_k \leq f\oplus \Id_n$ and $R_2\leq X
\oplus \Id_k \leq f\oplus \Id_n$, respectively. Since the $R_{1}$-equivalence
class of $0$ is not computable, its image under the reduction $R_1 \leq
f\oplus \Id_n$ must be in some column of $f$ (not in $\Id_n$). Let $W$ be the
set of elements whose image is not in the same column under the reduction
$R_1 \leq f\oplus \Id_n$. Since $[0]_{R_{1}}$ is not computable and it is not
intersected by $W$, we must have that $W$ contains only finitely many
$R_1$-classes by Remark \ref{rmk:co-finitelyMany}. Thus $R_1\leq_{\I}$ this one column of $f$. Similarly for
$R_2$. Since we did not explicitly code an edge between $R_1$ and $R_2$, there is no column of $f$ which is $\geq_{\I} R_1$ and $\geq_{\I} R_2$. So we see that $R_1\oplus R_2\leq_{\I}X$. Thus, since the $\I$-degree
of $X$ is assumed to be an $\I$-strongly minimal cover of the $\I$-degrees of
the pair $R_1,R_2$, we have $X\equiv_{\I} R_1\oplus R_2$. Thus there can only
be one $\I$-strongly minimal cover $\leq_{\I} \boldsymbol{f}$ of any pair of
minimal $\I$-degrees aside from the pairs where we intended to place an edge.
Thus there are no unwanted edges.

Therefore the $\I$-degree $\boldsymbol{f}$ of $f$ is the desired name for
$F$.	
\end{proof}

For $\I$-degrees $\boldsymbol{a}, \boldsymbol{b}, \boldsymbol{c}$ define
\[
[\boldsymbol{a}, \boldsymbol{b}]^{U^{\boldsymbol{c}}}
=\{\boldsymbol{x}\in U^{\boldsymbol{c}}\mid  \boldsymbol{a}
\leq^{\boldsymbol{c}} \boldsymbol{x} \leq^{\boldsymbol{c}} \boldsymbol{b}\}.
\]
This is clearly a ternary relation in $\boldsymbol{a}, \boldsymbol{b},
\boldsymbol{c}$, which is $\emptyset$-definable in $\Dark_{/\I}$.

\begin{defn}\label{def:tilde}
On pairs of dark $\I$-degrees we define the equivalence relation
$(\boldsymbol{c}, \boldsymbol{d}) \sim (\boldsymbol{c}',\boldsymbol{d}')$ if
the two pairs coincide, or $\boldsymbol{c}$ and $\boldsymbol{c}'$ are good
codes, $\boldsymbol{d}\in U^{\boldsymbol{c}}$ and $\boldsymbol{d}'\in
U^{\boldsymbol{c}'}$, and there exists a name $\boldsymbol{f}$ for a set of
pairs $F$ which is an order-preserving bijection between
$[\boldsymbol{0}^{\boldsymbol{c}},\boldsymbol{d}]^{U^{\boldsymbol{c}}}$ and
$[\boldsymbol{0}^{\boldsymbol{c}'},\boldsymbol{d}']^{U^{\boldsymbol{c}'}}$. That is, $G_{\boldsymbol{f}}$ is a graph so that for every $\boldsymbol{x}\in [\boldsymbol{0}^{\boldsymbol{c}},\boldsymbol{d}]^{U^{\boldsymbol{c}}}$, there exists a unique $\I$-minimal $\boldsymbol{y}$ so that there is a graph-label for the pair $(\boldsymbol{x},\boldsymbol{y})$ contained in $G_{\boldsymbol{f}}$. Further, this $\boldsymbol{y}$ is always in $[\boldsymbol{0}^{\boldsymbol{c}'},\boldsymbol{d}']^{U^{\boldsymbol{c}'}}$. Furthermore, this map that sends $\boldsymbol{x}\in [\boldsymbol{0}^{\boldsymbol{c}},\boldsymbol{d}]^{U^{\boldsymbol{c}}}$ to this unique $\boldsymbol{y}$ is an order-preserving bijection between $[\boldsymbol{0}^{\boldsymbol{c}},\boldsymbol{d}]^{U^{\boldsymbol{c}}}$ and
$[\boldsymbol{0}^{\boldsymbol{c}'},\boldsymbol{d}']^{U^{\boldsymbol{c}'}}$.
\end{defn}

\begin{lemma}\label{lem:sim-definable}
The relation $\sim$ is $\emptyset$-definable in the collection of dark $\I$-degrees.
\end{lemma}

\begin{proof}
This is straightforward from the definition of $\sim$ and Remark~\ref{rem:from-N-to-po}.
\end{proof}

Finally, we define:

\begin{defn}
Let $\mathcal{N}$ be the set of $\sim$-equivalence classes of pairs
$(\boldsymbol{c},\boldsymbol{d})$ of dark $\I$-degrees so that for every good
code $\boldsymbol{c}'$, there exists a $\boldsymbol{d}'$ so that
$(\boldsymbol{c}, \boldsymbol{d})\sim (\boldsymbol{c}',\boldsymbol{d}')$.
\end{defn}

\begin{lemma}
$\mathcal{N}$ is $\emptyset$-definable in the dark $\I$-degrees.
\end{lemma}

\begin{proof}
By Corollary~\ref{cor:good-definable} and Lemma~\ref{lem:sim-definable}.
\end{proof}

\begin{lemma}\label{lem:good-crucial}
Let $\boldsymbol{c}$ be a good code so that $(U^{\boldsymbol{c}},
\phi_+^{\boldsymbol{c}}, \phi_\times^{\boldsymbol{c}}) \simeq (\mathbb{N},
+,\times)$. Then $\mathcal{N}=\{[(\boldsymbol{c},\boldsymbol{d})]_\sim \mid
\boldsymbol{d}\in U^{\boldsymbol{c}}\}$ (where
$[(\boldsymbol{c},\boldsymbol{d})]_\sim$ denotes the equivalence class of the
pair $(\boldsymbol{c},\boldsymbol{d})$ under $\sim$).
\end{lemma}

\begin{proof}
Every model of Robinson's $Q$ has a standard part isomorphic to
$(\mathbb{N},+,\cdot)$. Note that if $R$ is in a degree in $U^{\boldsymbol
{c}'}$ for any good $\boldsymbol{c}'$, then $R$ must have a non-computable
class by Lemma \ref{ZsAreIsolated}. By Lemma \ref{namesExist}, this shows
that $\{[(\boldsymbol{c},\boldsymbol{d})]_\sim\mid \boldsymbol{d}\in
U^{\boldsymbol{c}}\}\subseteq \mathcal{N}$. For the converse, if
$[(\boldsymbol{c}',\boldsymbol{d}')]_\sim\in \mathcal{N}$, then by definition
of $\mathcal{N}$, there must be some $\boldsymbol{d}\in U^{\boldsymbol{c}}$
so that $(\boldsymbol{c}',\boldsymbol{d}')\sim(\boldsymbol{c},
\boldsymbol{d})$. Thus $\mathcal{N}\subseteq
\{[(\boldsymbol{c},\boldsymbol{d})]_\sim\mid \boldsymbol{d}\in
U^{\boldsymbol{c}}\}$.
\end{proof}

This allows us to define, without parameters, $+$ and $\times$ on
$\mathcal{N}$. For instance,

\begin{defn}
Let $\boldsymbol{c}$ be a good code so that $(U^{\boldsymbol{c}},
\phi_+^{\boldsymbol{c}}, \phi_\times^{\boldsymbol{c}}) \simeq (\mathbb{N},
+,\times)$. We define
$[(\boldsymbol{c},\boldsymbol{d})]_\sim+[(\boldsymbol{c},
\boldsymbol{d}')]_\sim=[(\boldsymbol{c},
\phi_+^{\boldsymbol{c}}(\boldsymbol{d},\boldsymbol{d}'))]_\sim$, and
similarly for the other operation.
\end{defn}

\begin{lemma}
The definition of $+$ and $\times$ on $\mathcal{N}$ does not depend on the
choice of $\boldsymbol{c}$ and is $\emptyset$-definable in the
partial order of $\Dark_{/\I}$ degrees.
\end{lemma}

\begin{proof}
The first claim is immediate from the definitions. We can define addition by
saying that $[(\boldsymbol{c},\boldsymbol{d})]_\sim
+[([\boldsymbol{c}',\boldsymbol{d}'])]_\sim =
[(\boldsymbol{c}'',\boldsymbol{d}'')]_\sim$ if and only if there exists a
good code $\hat{\boldsymbol{c}}$ and
$\boldsymbol{e},\boldsymbol{e}',\boldsymbol{e}''\in U^{\hat{c}}$ so that
$(\boldsymbol{c},\boldsymbol{d})\sim (\hat{\boldsymbol{c}},\boldsymbol{e})$,
$(\boldsymbol{c}',\boldsymbol{d}')\sim
(\hat{\boldsymbol{c}},\boldsymbol{e}')$,
$(\boldsymbol{c}'',\boldsymbol{d}'')\sim (\hat{c},\boldsymbol{e}'')$, and
that $\phi_+^{\hat{\boldsymbol{c}}}(\boldsymbol{e},\boldsymbol{e}')=
\boldsymbol{e}''$.		
\end{proof}

It is immediate that $(\mathcal{N},+,\times)$ is isomorphic to
$(\mathbb{N},+,\times )$, and thus we have proved:

\begin{thm}\label{thm:theory-of-I-dark}
There is an interpretation without parameters of $(\mathbb{N},+,\times)$ in the partial order
$\Dark_{/\I}$. Thus $\Th(\Dark_{/\I})$ is computably isomorphic to
$\Th^1(\mathbb{N})$.
\end{thm}

\begin{proof}
We have just shown that $\Th^1(\mathbb{N}) \leq_1 \Th(\Dark_{/\I})$. On the
other hand, $\Th(\Dark_{/\I}) \leq_1 \Th^1(\mathbb{N})$ by
Lemma~\ref{lem:one-way}.
\end{proof}

\begin{corollary}
$\Th(\Dark)$ and $\Th(\Ceers)$ are computably isomorphic to $\Th^1(\mathbb{N})$.
\end{corollary}
\begin{proof}
$\I$-equivalence is $\emptyset$-definable in $\Dark$ by Lemma~\ref{lem:I-definability} and $\Dark$ is $\emptyset$-definable in $\Ceers$ by \cite[Corollary 8.1]{Andrews-Sorbi}.
Thus we have $\Th^1(\mathbb{N}) \leq_1 \Th(\Dark_{/\I}) \leq_1
\Th(\Dark)\leq_1 \Th(\Ceers)\leq_1 \Th^1(\mathbb{N})$.
\end{proof}

Finally,

\begin{thm}
The theory of the partial order of $\Ceers_{/\I}$ is $1$-equivalent to the
theory of true arithmetic.
\end{thm}	

\begin{proof}
We use the same definition of $G_{\boldsymbol{c}}$ to code graphs and the same definition of $\sim$ to give the same definition of $(\mathbb{N},+,\cdot)$. The only subtlety in this case is that perhaps $\boldsymbol{\Id}$, the $\I$-degree of $\Id$, could be in $G_{\boldsymbol{c}}$ for some $\boldsymbol{c}$. Lemma \ref{namesExist} does not give us names for sets of pairs including $\boldsymbol{\Id}$. This is solved just as the case of $\mathbb{Z}$-dark minimal ceers with all computable classes. Namely, such elements must be isolated if they appear in $G_{\boldsymbol{c}}$.

By \cite[Observation 5.1 and Lemma 6.5]{Andrews-Sorbi}, $\boldsymbol{\Id}$ has a least upper bound with every $\I$-degree. Thus, it cannot have two $\I$-strongly minimal covers with any degree and thus if it appears in $G_{\boldsymbol{c}}$, it is isolated, and thus makes no difference in our definition of $(\mathbb{N},+,\cdot)$.
\end{proof}

\section{Coding graphs into the partial order $\Light_{/\I}$ using
parameters}\label{LightGraphs}

We now turn our attention to the structure $\Light_{/\I}$. We once again need to show that any computable graph can be encoded in $\Light_{/\I}$. To do this, we use the analogous result in $\Dark_{/\I}$ and an embedding $\iota$ of $\Dark_{/\I}$ into $\Light_{/\I}$.

Recall that the symbol $\bId$ denotes the $\I$-degree of the identity ceer $\Id$.

\begin{defn}
We say that a light $\I$-degree $\boldsymbol{e}$ is \emph{light minimal} if
it is $>_{\I} \bId$ and  $(\bId, \boldsymbol{e})_{\I}$ is empty, where
$(\bId, \boldsymbol{e})_{\I}$ is the interval of  $\I$-degrees
$\boldsymbol{x}$ such that $\bId <_{\I} \boldsymbol{x} <_{\I}
\boldsymbol{e}$.
\end{defn}

Note that the property of being light minimal is $\emptyset$-definable in the partial
order $\Light_{/\I}$. The following is Theorem 6.2 in \cite{Andrews-Sorbi}.

\begin{lemma}\label{lem:i-embed}
The map $\iota:X\mapsto X\oplus \Id$ induces an embedding of $\Dark_{/\I}$
into $\Light_{/\I}$.
\end{lemma}

\begin{proof}
For two dark ceers $X,Y$ we have
\begin{align*}
X\oplus \Id\leq_{\I} Y\oplus \Id
           &\Leftrightarrow X\oplus \Id \leq Y\oplus \Id\\
           &\Leftrightarrow X\leq Y\oplus \Id
\end{align*}
the last equivalence coming from that fact that $Y\oplus \Id\oplus \Id \equiv
Y\oplus \Id$ and thus if $X\leq Y\oplus \Id$ then $X\oplus \Id  \leq Y\oplus
\Id$. On the other hand, $X\leq Y\oplus \Id$ if and only if $X\leq Y\oplus
\Id_k$ for some $k$, because the darkness of $X$ guarantees that the
reduction cannot be infinite into the $\Id$ part. But since $X\leq Y\oplus
\Id_k$, for some $k$,  is just the definition of $X\leq_\I Y$, we see
\[
 X\leq_\I Y \Leftrightarrow X\oplus \Id\leq_{\I} Y\oplus \Id.
\]
\end{proof}

We will thus also use $\iota$ to refer to the induced map on $\I$-degrees.

\begin{lemma}\label{lem:i-onto}
If $X$ is dark, $Y$ is light, and $Y\leq X\oplus \Id$, then $Y \equiv \Id$ or
there is a dark ceer $Z\leq X$ so that $Y\equiv Z\oplus \Id$.
\end{lemma}

\begin{proof}
Assume $X, Y$ as in the statement of the lemma. Since $Y\leq X\oplus \Id$, by
Fact~\ref{fct:restriction} we have $Y\equiv Z \oplus Y_0$ where $Z\leq X$
(and thus $Z$ is either finite or dark) and $Y_0\leq \Id$. If $Y_0$ has only
finitely many classes, then $Y\equiv Z\oplus \Id_k$, but then $Z$ is dark and
so is its uniform join with a finite ceer, which is impossible since $Y$ is
light. It follows that $Y_{0} \equiv \Id$ and thus if $Z$ is finite then $Y
\equiv \Id$, otherwise $Y\equiv Z \oplus \Id$ for some dark $Z$.
\end{proof}

Thus, the ceers of the form $X\oplus \Id$ with $X$ dark form an initial
segment in the light ceers $>\Id$.

\begin{theorem}\label{thm:isomorphism-initial-segments}
Let $C$ be dark and let $\boldsymbol{c}$ denote its $\I$-degree. Then
$\Dark_{/\I}(\leq_\I \boldsymbol{c})\simeq (\bId, \iota (\boldsymbol{c}))_\I$.
\end{theorem}

\begin{proof}
Immediate from Lemma~\ref{lem:i-embed} (which shows that $\iota$ is an
order-theoretic embedding) and Lemma~\ref{lem:i-onto} which gives onto-ness.
\end{proof}

\begin{defn}
For any $\Light_{/\I}$ degree $\boldsymbol{c}$, we associate a graph
$H_{\boldsymbol{c}}$ as follows:
\begin{itemize}
\item (vertices) the \emph{vertices} are the light minimal degrees $\leq_\I
    \boldsymbol{c}$;
\item (edges) we place an \emph{edge} between vertices $\boldsymbol{d}$ and
    $\boldsymbol{e}$ if and only if there are two incomparable $\I$-degrees
    $\boldsymbol{a}, \boldsymbol{b}\leq_\I \boldsymbol{c}$ which are a light
    \ismc of the pair $\boldsymbol{d}$ and $\boldsymbol{e}$, i.e.
    $\boldsymbol{d}, \boldsymbol{e}\leq_\I \boldsymbol{a}, \boldsymbol{b}$
    and the only light $\I$-degrees $<_\I \boldsymbol{a}$ or $<_\I
    \boldsymbol{b}$ are $\boldsymbol{d}, \boldsymbol{e}$ and $\bId$.
\end{itemize}
\end{defn}

\begin{lemma}\label{lem:from-dark-min-to-light-min}
For any dark ceer $\boldsymbol{c}$, $H_{\iota(\boldsymbol{c})}\simeq
G_{\boldsymbol{c}}$.
\end{lemma}

\begin{proof}
This is immediate from Theorem~\ref{thm:isomorphism-initial-segments}: For a dark $\I$-degree $\boldsymbol{x}$, $\boldsymbol{x}$ is $\I$-minimal if and only if  $\iota(\boldsymbol{x})$ is light-minimal. Further, for dark degrees $\boldsymbol{x},\boldsymbol{y}, \boldsymbol{z}$, $\boldsymbol{z}$ is an \ismc of $\boldsymbol{x},\boldsymbol{y}$ if and only if  $\iota(\boldsymbol{z})$ is a light \ismc of $\iota(\boldsymbol{x}),\iota(\boldsymbol{y})$.
\end{proof}

\begin{corollary}
For any computable graph $G$, there is a light $\I$-degree $\boldsymbol{c}$
so that $H_{\boldsymbol{c}}$ is isomorphic to the disjoint union of $G$ with a graph which has no edges.
\end{corollary}

\begin{proof}
By Theorem~\ref{codesExist} and Lemma~\ref{lem:from-dark-min-to-light-min}.
\end{proof}

\section{Interpreting $(\mathbb{N},+,\times)$ in the partial order
$\Light_{/\I}$ without parameters}\label{LightWithoutParameters}

Following what we did for dark degrees in Definition~\ref{def:good-code} (and
using the notations therein exploited), we define a light $\I$-degree
$\boldsymbol{c}$ to be \emph{good} if in $H_{\boldsymbol{c}}$, the triple
$(U^{\boldsymbol{c}},\phi^{\boldsymbol{c}}_+,\phi^{{\boldsymbol{c}}}_\times)$
gives a model of Robinson's $Q$. Recall that $\iota$ is an initial embedding of $\Dark/\I$ into $\Light/\I\smallsetminus\{\boldsymbol{\Id}\}$. So, we have a good code $\boldsymbol{c}$ so that $H_{\boldsymbol{c}}$ codes $(\mathbb{N},+,\cdot)$. Our goal is to define $\mathcal{N}$ as
in Section~\ref{sct:dark-arithmetic}, but we have to use a different coding
for finite functions. The reason is that we may have good codes which are
$\Light_{/\I}$-degrees which are not in the image of $\iota$ (the embedding introduced in Lemma~\ref{lem:i-embed}),
thus we cannot
use the $\iota$-image of the construction for names in $\Dark_{/\I}$.

Throughout the section an \emph{$\I$-strongly minimal} cover of an $\I$
degree $\boldsymbol{x}$ means an $\I$ degree $\boldsymbol{y} >_{\I}
\boldsymbol{x}$ such that the interval $[\bId, \boldsymbol{y})_{\I}$ is
exactly the interval $[\bId,\boldsymbol{x}]_{\I}$. Since we will consider
only $\I$-strongly minimal covers of light $\I$-degrees, they will be light
as well.

\begin{defn}
Let $F=\{\{\boldsymbol{a}_i,\boldsymbol{b}_i\}\mid i\in S\}$ be a set of unordered
pairs of light minimal $\I$-degrees so that the  $\boldsymbol{a}_i$'s and the
$\boldsymbol{b}_j$'s are distinct (including $\boldsymbol{a}_i\neq
\boldsymbol{b}_j$ for any $i,j \in S$). We say that a light $\I$-degree
$\boldsymbol{f}$ is a \emph{name for $F$} if the only light minimal
$\I$-degrees below $\boldsymbol{f}$ are $\{\boldsymbol{a}_i,
\boldsymbol{b}_i\mid i\in S\}$, and the $\{ \boldsymbol{a}_i,
\boldsymbol{b}_i \}$'s are the only pairs $\{ \boldsymbol{c},\boldsymbol{d}
\}$ of light minimal $\I$-degrees less than $\boldsymbol{f}$ for which there
is an $\boldsymbol{x}<_\I \boldsymbol{f}$ so the only light minimal
$\I$-degrees less than $\boldsymbol{x}$ are $\boldsymbol{c}$ and
$\boldsymbol{d}$, and $\boldsymbol{x}$ has a light \ismc $\boldsymbol{y}$, which
in turn has a light \ismc $\boldsymbol{z}$ which is $\leq_\I \boldsymbol{f}$.
\end{defn}

\begin{lemma}\label{namesExistLight}
Let $F=\{\{ \boldsymbol{a}_i, \boldsymbol{b}_i \}\mid i<n\}$ be a finite set
of pairs of light minimal $\I$-degrees, so that the $\boldsymbol{a}_i$'s and
the $\boldsymbol{b}_j$'s are distinct. Then there is a name for $F$.
\end{lemma}

\begin{proof}
In this proof we use that in $\Ceers_{/\I}$ every non-universal element has
infinitely many distinct self-full strong minimal covers, see
\cite[Theorem~7.9]{Andrews-Sorbi}. For each pair $\{\boldsymbol{a}_i,
\boldsymbol{b}_i \} \in F$, we let $\boldsymbol{c}_i$ be a self-full
$\I$-strongly minimal cover of $\boldsymbol{a}_i\oplus \boldsymbol{b}_i$. Let
$\boldsymbol{d}_i$ be a self-full $\I$-strongly minimal cover of
$\boldsymbol{c}_i$. Note that we choose these to be $\I$-strongly minimal
covers in $\Ceers_{/\I}$, not just in $\Light_{/\I}$. Let
$\boldsymbol{f}=\bigoplus_i \boldsymbol{d}_i$.

First we check that for each
pair $\{ \boldsymbol{a}_i,\boldsymbol{b}_i \}$ in $F$, there is an
$\boldsymbol{x}_i<_\I \boldsymbol{f}$ (take $\boldsymbol{x}_i=
\boldsymbol{a}_i \oplus \boldsymbol{b}_i$) so that the only light minimal
$\I$-degrees less than $\boldsymbol{x}_i$ are $\boldsymbol{a}_i$ and
$\boldsymbol{b}_i$, and $\boldsymbol{x}_i$ has a light \ismc $\boldsymbol{y}_i$
(take $\boldsymbol{y}_i=\boldsymbol{c}_i$) which has a light \ismc
$\boldsymbol{z}_i \leq_\I \boldsymbol{f}$ (take
$\boldsymbol{z}_i=\boldsymbol{d_i}$). To see that the only light minimal
$\I$-degrees $\leq_\I \boldsymbol{x}_i$ are $\boldsymbol{a}_i$ and
$\boldsymbol{b}_i$, assume that $A_i \in \boldsymbol{a}_i$, $B_i \in
\boldsymbol{b}_i$, and $X_i=A_i \oplus B_i$. If $U \leq_\I X_i$ has light
minimal $\I$-degree then by Fact~\ref{fct:restrictionI} there exist $U_0,
U_1$ such that $U \equiv_\I U_0 \oplus U_1$ with $U_0 \leq A_i$ and $U_1 \leq
B_i$. By light minimality of $A_i, B_i$ it follows
that
$U_0$ is
either finite, dark, $U_0 \equiv_{\I} A_i$, or $U_0 \equiv_\I \Id$, and similarly
$U_1$ is either finite, dark, $U_1 \equiv_\I B_i$, or $U_1 \equiv_\I \Id$. If
$U_0\equiv_\I A_i$, then we see that $A_i\leq_{\I} U$, so by light minimality
of $U$ and $A_i$, we have that $U\equiv_{\I}A_i$. Similarly if $U_1\equiv_\I
B_i$ then $U\equiv_{\I} B_i$. Thus we can assume neither of these cases
holds. By lightness of $U$, it follows at least one of $U_0$ or $U_1$ is
light, so without loss of generality, we suppose $U_0$ is light, i.e.
$U_0\equiv_{\I} \Id$. If $U_1$ is finite or $\equiv_{\I} \Id$, then
$U\equiv_{\I}\Id$, contradicting $U$ being of light minimal $\I$-degree. Thus
$U_1$ must be dark $\leq B_i$. So, $U\equiv_\I \Id\oplus D$ for some dark
$D\leq B_i$. But since every dark ceer $D$ has a join with $\Id$, namely
$D\oplus \Id$ \cite[Obs 5.1]{Andrews-Sorbi}, it follows that $U\equiv_\I
D\oplus \Id\leq B_i$. Again, by light minimality of $U$ and $B_i$, we see
$U\equiv_{\I} B_i$.

It remains to show that no other pair $\leq_\I \boldsymbol{f}$ has such a
triple $\boldsymbol{x},\boldsymbol{y}, \boldsymbol{z}$. We begin with an easy
observation about the $\I$-degrees $\leq \boldsymbol{d}_i$. Together with
$A_i \in \boldsymbol{a}_i$, $B_i \in \boldsymbol{b}_i$, fix also
representatives $C_i \in \boldsymbol{c}_i$, $D_i \in \boldsymbol{d}_i$, and
$f \in \boldsymbol{f}$. We also use the notation $\boldsymbol{0}_\I$ to
denote the least $\I$-degree, which is the $\equiv_\I$-class of the ceer
$\Id_1$ with only one equivalence class, and is comprised exactly of all
finite ceers.

Next we show that for each $\boldsymbol{a}_i$ there can be
at most one dark $\I$ degree $\boldsymbol{r}_i \leq_{\I}\boldsymbol{a}_i$. To
see this assume towards a contradiction that $\boldsymbol{r}, \boldsymbol{s}$ are distinct dark $\I$-degrees
below $\boldsymbol{a}_i$. By \cite[Obs 5.1 and Lemma 6.5]{Andrews-Sorbi}, $\boldsymbol{r} \oplus \bId$ and
$\boldsymbol{s} \oplus \bId$ are joins of the two $\I$-degrees, and thus they
are both $\leq_{\I} \boldsymbol{a}_i$. Since $\iota$ is an embedding, $\boldsymbol{r} \oplus \bId$ and $\boldsymbol{s} \oplus
\bId$ are distinct light degrees in $(\bId,\boldsymbol{a}_i]$. This contradicts the light minimality of $\boldsymbol{a}_i$. We conclude that there can be only one dark $\I$-degree $\leq_{\I}\boldsymbol{a}_i$. Further, if there is a
(unique) dark $\I$-degree $\boldsymbol{r}_i \leq_{\I} \boldsymbol{a}_i$ then
$\boldsymbol{a}_i = \boldsymbol{r}_i \oplus \bId$. In the following if there
is no dark degree below $\boldsymbol{a}_i$ then we let
$\boldsymbol{r}_i=\boldsymbol{a}_i$. Similar considerations and notations
hold for each $\boldsymbol{b}_i$. In particular, if there is a (unique) dark
degree $\I$-degree $\leq \boldsymbol{b}_i$, we call it $\boldsymbol{s}_i$,
and if there is no dark $\boldsymbol{s}_i\leq_{\I}\boldsymbol{b}_i$, then we
let $\boldsymbol{s}_i=\boldsymbol{b}_i$.

\begin{claim}
If $\boldsymbol{w}\leq_\I \boldsymbol{d}_i$ is an $\I$-degree, then
$\boldsymbol{w} \in\{\boldsymbol{0}_\I, \bId, \boldsymbol{a}_i,
\boldsymbol{b}_i, \boldsymbol{r}_i, \boldsymbol{s}_i, \boldsymbol{r}_i\oplus
\boldsymbol{b}_i, \boldsymbol{a}_i\oplus \boldsymbol{s}_i,
\boldsymbol{a}_i\oplus \bId, \boldsymbol{b}_i\oplus \bId,
\boldsymbol{a}_i\oplus \boldsymbol{b}_i, \boldsymbol{c}_i,
\boldsymbol{d}_i\}$.
\end{claim}

\begin{proof}
If $W\leq_\I D_i$, then $W\equiv_\I D_i$ or $W\equiv_\I C_i$ or $W\leq_{\I}
A_i\oplus B_i$ since $\boldsymbol{d}_i$ is an \ismc of $\boldsymbol{c}_i$
which is an \ismc of $\boldsymbol{a}_i\oplus \boldsymbol{b}_i$. If
$W\leq_{\I} A_i\oplus B_i$, then by Fact~\ref{fct:restrictionI} $W\equiv_\I
W_0\oplus W_1$ where $W_0\leq A_i$ and $W_1\leq B_i$. Since each
$\boldsymbol{a}_i$ and $\boldsymbol{b}_i$ is light minimal, the only
possibilities for the $\I$-degree of $W_0$ are $\boldsymbol{a}_i$,
$\boldsymbol{r}_i$, $\bId$, or $\boldsymbol{0}_\I$, and the only
possibilities for the $\I$-degree of $W_1$ are $\boldsymbol{b}_i$,
$\boldsymbol{s}_i$, $\bId$, or $\boldsymbol{0}_{\I}$. The possible direct
sums of these are easily seen to be the possibilities listed in the claim.
\end{proof}

Suppose that the only light minimal $\I$-degrees bounded by $\boldsymbol{x}$ are
$\boldsymbol{a}_i$ and $\boldsymbol{a}_j$, with $i\ne j$ (it is no different
if we consider $\boldsymbol{a}_i$ and $\boldsymbol{b}_j$ or
$\boldsymbol{b}_i$ and $\boldsymbol{b}_j$). Since $X\leq_\I
f\equiv_\I\bigoplus D_r$, we see by Fact~\ref{fct:restrictionI} that
$X\equiv_\I \bigoplus_r X_r$ where each $X_r\leq D_r$. Thus each $X_r$ is
$\I$-equivalent to one of the $\I$-degrees in the list: $\boldsymbol{0}_\I,
\bId, \boldsymbol{a}_r, \boldsymbol{b}_r, \boldsymbol{r}_r, \boldsymbol{s}_r,
\boldsymbol{r}_r\oplus \boldsymbol{b}_r, \boldsymbol{a}_r \oplus
\boldsymbol{s}_r, \boldsymbol{a}_i\oplus \bId, \boldsymbol{b}_i\oplus \bId,
\boldsymbol{a}_r\oplus \boldsymbol{b}_r, \boldsymbol{c}_r, \boldsymbol{d}_r$.
Since $\boldsymbol{x}$ only bounds $\boldsymbol{a}_i$ and $\boldsymbol{a}_j$,
for the $\I$-degrees of the various $X_r$ we must rule out the possibilities
that any such degree  bounds in the above list a light minimal $\I$-degree
not in $\{\boldsymbol{a}_i, \boldsymbol{a}_j\}$; finally we can remove from
the list the $\I$-degrees $\boldsymbol{a}_i \oplus \boldsymbol{s}_i$, and
$\boldsymbol{a}_j \oplus \boldsymbol{s}_j$: note for example that
$\boldsymbol{a}_i\oplus \boldsymbol{s}_i\geq_{\I} \bId\oplus \boldsymbol{s}_i
= \boldsymbol{b}_i$, so this one cannot be the $\I$-degree of $X_i$, unless
$\boldsymbol{b}_{i} \in \{\boldsymbol{a}_i, \boldsymbol{a}_j\}$. Then we
conclude that the $\I$-degree of any of the $X_{r}$'s is equal to one of the
$\I$-degrees in the following list: $\boldsymbol{0}_{\I}$, $\bId$,
$\boldsymbol{a}_i$, $\boldsymbol{r}_k$, $\boldsymbol{s}_k$,
$\boldsymbol{a}_i\oplus \bId$, $\boldsymbol{a}_j$, $\boldsymbol{a}_j\oplus
\bId$, for some $k$.

We write $\boldsymbol{a}_i$ as either $\boldsymbol{r}_i$ or
$\boldsymbol{r}_i\oplus \bId$, depending on which is true. Similarly for
$\boldsymbol{a}_j$. It then follows that the only possibility for
$\boldsymbol{x}$ is to be of the form $\boldsymbol{r}_i\oplus
\boldsymbol{r}_j (\oplus \boldsymbol{r}_k)(\oplus \boldsymbol{s}_k)(\oplus
\bId)$ where parentheses are meant to symbolize that we may or may not be
joining with these degrees. But we observe that any degree $\boldsymbol{u}
\geq_\I \boldsymbol{a}_i$ and $\geq_{\I} \boldsymbol{r}_k$ for any $k$, is
also $\geq_{\I} \boldsymbol{a}_k$. This is because $\boldsymbol{u}$ is light, and thus $\geq_{\I}\boldsymbol{a}_k$, since $\boldsymbol{a}_k$ is the join of $\bId$ and $\boldsymbol{r}_k$. Since
$\boldsymbol{x}$ bounds only $\boldsymbol{a}_i$ and $\boldsymbol{a}_j$, this
is impossible if $k \notin\{i,j\}$. Thus, we are left with the only
possibilities for $\boldsymbol{x}$ being $\boldsymbol{r}_i\oplus
\boldsymbol{r}_j(\oplus\bId)$. Now, $\boldsymbol{x}$ bounded
$\boldsymbol{a}_i$ and $\boldsymbol{a}_j$ and no other light minimal
$\I$-degrees, and $\boldsymbol{y}$ was an $\I$-strongly minimal cover of
$\boldsymbol{x}$, then it would also bound no other light minimal degrees,
and $\boldsymbol{z}$ being an $\I$-strongly minimal cover of $\boldsymbol{y}$
would also bound no other light minimal $\I$-degrees. On the other hand, the
same argument as before applied to $\boldsymbol{y}$ and $\boldsymbol{z}$
would show that also $\boldsymbol{y}$ and $\boldsymbol{z}$ may only have the
form $\boldsymbol{r}_i\oplus \boldsymbol{r}_j(\oplus\bId)$. But there are at
most two such $\I$-degrees, contradicting the fact that $\boldsymbol{x},
\boldsymbol{y}, \boldsymbol{z}$ are three distinct $\I$-degrees.

Thus $\boldsymbol{f}$ is a name for $F$.
\end{proof}	

\begin{defn}
Given two good codes $\boldsymbol{c}, \boldsymbol{c}'$ for graphs in
$\Light_{/\I}$ (thus giving models of Robinson's $Q$), we say that a pair of names
$(\boldsymbol{f}, \boldsymbol{g})$ is a \emph{label} for a partial function
$H: H_{\boldsymbol{c}} \longrightarrow H_{\boldsymbol{c}'}$ if
$\boldsymbol{f}$ is a name for a set $F$ and $\boldsymbol{g}$ is a name for a
set $G$ and whenever $\boldsymbol{a}<_\I \boldsymbol{c}$ is light minimal,
then there is a light minimal $\I$-degree $\boldsymbol{b}$ so that $\{
\boldsymbol{a},\boldsymbol{b} \}$ is a pair in $F$ and
$\boldsymbol{b}\not\leq_\I \boldsymbol{c},\boldsymbol{c}'$, and $\{
\boldsymbol{b}, H(\boldsymbol{a})\}$ is a pair in $G$.
\end{defn}

\begin{lemma}
If $H$ is a finite function between $H_{\boldsymbol{c}}$ and
$H_{\boldsymbol{c}'}$, with $\boldsymbol{c}, \boldsymbol{c}'$ light
$\I$-degrees, there exists a pair $(\boldsymbol{f},\boldsymbol{g})$ which is
a label for this function.
\end{lemma}

\begin{proof}
We need to use light minimal degrees which are not below
$\boldsymbol{c}\oplus \boldsymbol{c}'$ to interpolate for the function. Such
degrees exist, because there are infinitely many dark minimal degrees
avoiding any lower cone \cite[Theorem~3.3]{Andrews-Sorbi}. By
Lemma~\ref{lem:i-embed} we just use the $\iota$ image of these. If $C,C'$ are
representatives of $\boldsymbol{c}, \boldsymbol{c}'$ respectively, and
$D\nleq C\oplus C'$ is dark minimal, then $\iota(D)\nleq_{\I} C\oplus C'$
because $D\nleq_\I C\oplus C'$, by Observation
\ref{inseparabilityDarkMinimal}

The existence of the names for the needed sets of pairs is then given by
Lemma \ref{namesExistLight}.
\end{proof}

\begin{defn}
On $\I$-degrees we define the equivalence relation
$(\boldsymbol{c},\boldsymbol{d})\sim(\boldsymbol{c}',\boldsymbol{d}')$ if the
two pairs coincide, or $\boldsymbol{c}$ and $\boldsymbol{c}'$ are good codes,
$\boldsymbol{d}\in U^{\boldsymbol{c}}$ and $\boldsymbol{d}'\in
U^{\boldsymbol{c}'}$ and there exists a pair of names
$(\boldsymbol{f},\boldsymbol{g})$ which is a label for a function $H$ which
is an order-preserving bijection between $[\boldsymbol{0}^{\boldsymbol{c}},
\boldsymbol{d}]^{U^{\boldsymbol{c}}}$ and
$[\boldsymbol{0}^{\boldsymbol{c}'},\boldsymbol{d}']^{U^{\boldsymbol{c}'}}$,
where the various symbols $U^{\boldsymbol{c}}$, $U^{\boldsymbol{c}'}$,
$\boldsymbol{0}^{\boldsymbol{c}}$ and $\boldsymbol{0}^{\boldsymbol{c}'}$ have
the same meanings as in Section~\ref{sct:dark-arithmetic}.
\end{defn}

This is what is needed to again define $\mathbb{N}$ exactly as in the dark
case, as explained in Section~\ref{sct:dark-arithmetic}.

Thus we have shown that there is a copy of $(\mathbb{N},+,\times)$ interpretable without parameters
in the structure $\Light_{/\I}$, proving the following theorem:

\begin{thm}
The first order theory of $\Light_{/\I}$ is computably isomorphic to true
first order  arithmetic.
\end{thm}

\begin{corollary}
The first order theory of $\Light$ is computably isomorphic to true first
order  arithmetic.
\end{corollary}

\begin{proof}
$\I$-equivalence on light $\I$-degrees is $\emptyset$-definable in the light degrees, by
Lemma~\ref{lem:I-definability}.
\end{proof}

\section{Open Questions}

The argument in \cite{Andrews-et-al} showed that the 3-quantifier theory of $\Ceers$ is already undecidable, but does not clarify its degree. We ask:

\begin{question}
	What is the degree of the 3-quantifier theory of $\Ceers$?
\end{question}

\begin{question}
	What is the least $n$ so that the $n$-quantifier theory of $\Ceers$ is undecidable?
\end{question}

Our interpretation of $(\mathbb{N},+,\cdot)$ shows that there is some $k$ (for example any $k$ so that the interpretation is $\exists_k$
in the hierarchy of formulas, as described for instance in \cite[p.47f]{Hodges}) so that for every $m$, the $k+m$-quantifier theory of $\Ceers$ is $\geq_T \emptyset^{(m)}$. We ask:

\begin{question}
	What is the least such $k$? Is it true that the $k+m$-quantifier theory of $\Ceers$ is $\equiv_T \emptyset^{(m)}$?
\end{question}

In analyzing the $\mathbb{Z}$-dark minimal ceers, we concluded that the $\I$-degrees of a $\mathbb{Z}$-dark minimal ceer all of whose classes are computable have joins with every other $\I$-degree.

In \cite{Andrews-Sorbi}, we showed that $\Id$ is $\emptyset$-definable in $\Ceers$ as the only ceer which is minimal over the finite ceers and has a join with every other ceer. Lemma \ref{joinsforZdarkmins} shows that in $\Ceers_{/\I}$, any $\mathbb{Z}$-dark minimal ceer with all  computable classes also has this property. We ask:

\begin{question}
	Is $\bId$ $\emptyset$-definable in $\Ceers_{/\I}$?
	
	Is the collection of $\I$-degrees of dark minimal ceers definable (even with parameters)? That is, is there a definable way to distinguish between dark minimal and $\mathbb{Z}$-dark minimal $\I$-degrees?
\end{question}


\begin{thebibliography}{10}

\bibitem{KazakhPaper}
U.~Andrews and S.~Badaev.
\newblock On isomorphism classes of computably enumerable equivalence
  relations.
\newblock {\em J. Symbolic Logic}, 2019.
\newblock DOI 10.1017/jsl.2019.39.

\bibitem{Andrews-et-al}
U.~Andrews, S.~Lempp, J.S. Miller, K.M. Ng, L.~San~Mauro, and A.~Sorbi.
\newblock Universal computably enumerable equivalence relations.
\newblock {\em J. Symbolic Logic}, 79(1):60--88, 2014.

\bibitem{jumps}
U.~Andrews and A.~Sorbi.
\newblock Jumps of computably enumerable equivalence relations.
\newblock {\em Ann.\ Pure Appl.\ Logic}, 169:243--259, 2018.

\bibitem{Andrews-Sorbi}
U.~Andrews and A.~Sorbi.
\newblock Joins and meets in the structure of ceers.
\newblock {\em Computability}, 8(3-4):193--241, 2019.

\bibitem{Becker-Kechris:descriptive}
H.~Becker and A.~S. Kechris.
\newblock {\em The Descriptive Set Theory of Polish Group Actions}, volume 232
  of {\em London Mathematical Society Lecture Notes Series}.
\newblock Cambridge University Press, 1996.

\bibitem{Bernardi-Sorbi}
C.~Bernardi and A.~Sorbi.
\newblock Classifying positive equivalence relations.
\newblock {\em The Journal of Symbolic Logic}, 48(3):529--538, 1983.

\bibitem{CAi-totality}
M.~Cai, H.A. Ganchev, S.~Lempp, J.S. Miller, and M.I. Soskova.
\newblock Defining totality in the enumeration degrees.
\newblock {\em J. Amer. Math. Soc.}, 29(4):1051--1067, 2016.

\bibitem{Carroll}
J.S. Carroll.
\newblock Some undecidability results for lattices in recursion theory.
\newblock {\em Pacific J. Math.}, 122(2):319--331, 1986.

\bibitem{Ershov:positive}
Yu.~L. Ershov.
\newblock Positive equivalences.
\newblock {\em Algebra and Logic}, 10(6):378--394, 1973.

\bibitem{Ershov:NumberingsI}
Yu.~L. Ershov.
\newblock Theorie der {N}umerierungen {I}.
\newblock {\em Z. Math. Logik Grundlag. Math.}, 19:289--388, 1973.

\bibitem{Fokina-et-al-several}
E.B. Fokina, S.D. Friedman, V.~Harizanov, J.F. Knight, C.~McCoy, and
  A.~Montalb{\'a}n.
\newblock Isomorphism relations on computable structures.
\newblock {\em J. Symbolic Logic}, 77(1):122--132, 2012.

\bibitem{Fokina2010effective}
E.B. Fokina, S.D. Friedman, and A.~T{\"o}rnquist.
\newblock The effective theory of borel equivalence relations.
\newblock {\em Annals of Pure and Applied Logic}, 161(7):837--850, 2010.

\bibitem{fokina2016linear}
E.B. Fokina, B.~Khoussainov, P.~Semukhin, and D.~Turetsky.
\newblock Linear orders realized by ce equivalence relations.
\newblock {\em J. Symbolic Logic}, 81(2):463--482, 2016.

\bibitem{Ganchev-Soskova1}
H.~Ganchev and M.~Soskova.
\newblock Interpreting true arithmetic in the local structure of the
  enumeration degrees.
\newblock {\em J. Symbolic Logic}, 77(4):1184--1194, 2012.

\bibitem{Ganchev-Soskova2}
H.A. Ganchev and M.I. Soskova.
\newblock Definability via {K}alimullin pairs in the structure of the
  enumeration degrees.
\newblock {\em Trans. Amer. Math. Soc.}, 367(7):4873--4893, 2015.

\bibitem{Gao-Gerdes}
S.~Gao and P.~Gerdes.
\newblock Computably enumerable equivalence relations.
\newblock {\em Studia Logica}, 67(1):27--59, 2001.

\bibitem{gavruskin2014graphs}
A.~Gavruskin, S.~Jain, B.~Khoussainov, and F.~Stephan.
\newblock Graphs realised by r.e. equivalence relations.
\newblock {\em Annals of Pure and Applied Logic}, 165(7):1263--1290, 2014.

\bibitem{Hodges}
W.~Hodges.
\newblock {\em Model Theory}.
\newblock Encyclopedia of Mathematics and its Applications. Cambridge
  University Press, Cambridge, 1993.

\bibitem{Korec}
I.~Korec.
\newblock A list of arithmetical structures complete with respect to the
  first-order definability.
\newblock {\em Theoretical Computer Science}, 257:115--151, 2001.

\bibitem{Miller}
C.F. Miller~III.
\newblock {\em On Group-Theoretic Decision Problems and Their
  Classification.(AM-68)}, volume~68.
\newblock Princeton university press, Princeton, New Jersey, 1971.

\bibitem{Montagna-ufp}
F.~Montagna.
\newblock Relatively precomplete numerations and arithmetic.
\newblock {\em Journal of Philosophical Logic}, 11(4):419--430, 1982.

\bibitem{Nerode-Shore}
A.~Nerode and R.A. Shore.
\newblock Second order logic and first order theories of reducibility
  orderings.
\newblock In {\em The {K}leene {S}ymposium ({P}roc. {S}ympos., {U}niv.
  {W}isconsin, {M}adison, {W}is., 1978)}, volume 101 of {\em Stud. Logic
  Foundations Math.}, pages 181--200. North-Holland, Amsterdam-New York, 1980.

\bibitem{Nies-last-question}
A.~Nies.
\newblock The last question on recursively enumerable m-degrees.
\newblock {\em Algebra and Logic}, 33(5):307--314, 1994.

\bibitem{Nies-eq-rel-mod-finite}
A.~Nies.
\newblock Recursively enumerable equivalence relations modulo finite
  differences.
\newblock {\em Math. Logic Quart.}, 40(4):490--518, 1994.

\bibitem{Nies-Shore-Slaman}
A.~Nies, R.A. Shore, and T.A. Slaman.
\newblock Interpretability and definability in the recursively enumerable
  degrees.
\newblock {\em Proc. London Math. Soc. (3)}, 77(2):241--291, 1998.

\bibitem{Nies-Sorbi}
A.~Nies and A.~Sorbi.
\newblock Calibrating word problems of groups via the complexity of equivalence
  relations.
\newblock {\em Mathematical Structures in Computer Science}, pages 1--15, 2018.

\bibitem{Rogers:Book}
H.~Rogers, Jr.
\newblock {\em Theory of Recursive Functions and Effective Computability}.
\newblock McGraw-Hill, New York, 1967.

\bibitem{Shore}
R.A. Shore.
\newblock The theory of the degrees below $0'$.
\newblock {\em J.\ London Math.\ Soc.}, 24:1--14, 1981.

\bibitem{Simpson}
S.G. Simpson.
\newblock First order theory of the degrees of recursive unsolvability.
\newblock {\em Ann.\ of Math.}, 105:121--139, 1977.

\bibitem{Slaman-Woodin1}
T.A. Slaman and W.H. Woodin.
\newblock Definability in the {T}uring degrees.
\newblock {\em Illinois J.\ Math.}, 30:320--334, 1986.

\bibitem{Slaman-Woodin2}
T.A. Slaman and W.H. Woodin.
\newblock Definability in the enumeration degrees.
\newblock {\em Arch.\ Math.\ Logic}, 36:225--267, 1997.

\bibitem{Soare:Book}
R.~I. Soare.
\newblock {\em Recursively Enumerable Sets and Degrees}.
\newblock Perspectives in Mathematical Logic, Omega Series. Springer-Verlag,
  Heidelberg, 1987.

\end{thebibliography}

\end{document}